\documentclass[12pt,reqno]{amsart}
\usepackage[top=1in,bottom=1in,left=1.25in,right=1.25in]{geometry}
\usepackage{amsmath,amssymb}
\usepackage{color}
\usepackage{amsfonts}
\usepackage{mathtools}
\usepackage{setspace}
\usepackage[toc,page]{appendix}

\usepackage{hyperref}
\usepackage{cleveref}
\usepackage{enumitem}

\usepackage{amscd}
\usepackage{cjhebrew}
\usepackage[T1]{fontenc} 
\usepackage[utf8]{inputenc}
\usepackage{verbatim}
\usepackage{stmaryrd}
\usepackage{bbm}
\usepackage{mathrsfs}
\usepackage{todonotes}
\usepackage{thmtools}

\usepackage{enumerate}

\DeclareMathOperator{\sgn}{sgn}

\newtheorem{theorem}{Theorem}[section]
\newtheorem{proposition}[theorem]{Proposition}

\newtheorem{lemma}[theorem]{Lemma}

\theoremstyle{definition}
\newtheorem{definition}[theorem]{Definition}

\theoremstyle{remark}
\newtheorem{remark}[theorem]{Remark}

\numberwithin{equation}{section}

\renewcommand{\Re}{\operatorname{Re}}

\newcommand{\la}{\lambda}

\def\N{{\mathbb N}}

\def\R{{\mathbb R}}

\newcommand{\norm}[1]{\left\|#1\right\|}

\newcommand{\eps}{\varepsilon}

\title[]{Non-uniqueness of mild solutions to supercritical heat equations}

\author{Irfan Glogi\'c}
\address{Fakult\"at f\"ur Mathematik, Universit\"at Bielefeld, D-33501 Bielefeld, Germany}
\email{irfan.glogic@uni-bielefeld.de}

\author{Martina Hofmanov\'a}
\address{Fakult\"at f\"ur Mathematik, Universit\"at Bielefeld, D-33501 Bielefeld, Germany}
\email{hofmanova@math.uni-bielefeld.de}

\author{Theresa Lange}
\address{Scuola Normale Superiore Pisa, Piazza dei Cavalieri 7, 56126 Pisa, Italy}
\email{theresa.lange@sns.it}

\author{Eliseo Luongo}
\address{Fakult\"at f\"ur Mathematik, Universit\"at Bielefeld, D-33501 Bielefeld, Germany}
\email{eluongo@math.uni-bielefeld.de}

\thanks{I.G.~acknowledges support by the Austrian Science Fund FWF, Projects P34378 and PAT5825523.
	M.H.~and E.L.~are grateful for funding from the European Research Council (ERC) under the European Union’s Horizon 2020 research and innovation programme (grant agreement No. 949981) and for the financial support provided by the Deutsche Forschungsgemeinschaft (DFG, German Research Foundation) – Project-ID 317210226–SFB 1283. T.L.~has received funding from the European Research Council (ERC) under the EU-HORIZON EUROPE ERC-2021-ADG research and innovation programme (project ”Noise in Fluids“, grant agreement no.~101053472).
}

\usepackage{kantlipsum}

\usepackage{fullpage}
\allowdisplaybreaks
\begin{document}

	\begin{abstract}
		We consider the focusing power nonlinearity heat equation
		\begin{equation}\label{Eq:Heat_abstract}\tag{NLH}
			\partial_t u -\Delta u =  |u|^{p-1}u, \quad p>1, \\
		\end{equation}
		in dimensions $d \geq 3$. It is well-known that if $p$ is large enough
		then \eqref{Eq:Heat_abstract} is unconditionally locally well-posed in $L^q(\R^d)$ for $q \geq d(p-1)/2$. We prove that this result is optimal in the sense that uniqueness of local solutions fails when $q < d(p-1)/2$ as long as $p < p_{JL}$, where $p_{JL}$ stands for the Joseph-Lundgren exponent. Our proof is based on the method that Jia-\v{S}ver\'ak proposed in \cite{JiaSve15} to show non-uniqueness of Leray solutions to incompressible 3d Navier-Stokes equations. In particular, we rigorously verify for \eqref{Eq:Heat_abstract} the (analogue of the) spectral assumption made in \cite{JiaSve15}. To our knowledge, this is the first rigorous implementation of the Jia-\v{S}ver\'ak method to a nonlinear parabolic equation without forcing.
	\end{abstract}
	
	\maketitle
	\maketitle
	\section{Introduction}
	
	\noindent Consider the Cauchy problem for the focusing power nonlinearity heat equation
	\begin{equation}\label{Eq:Heat}
		\begin{cases}
			\partial_t u -\Delta u =  |u|^{p-1}u, \\
			u(0,\cdot)=u_0,
		\end{cases}
	\end{equation}
	where $u=u(t,x)\in \R$, $(t,x) \in [0,\infty) \times \R^d$, $d \geq 3$ and $p>1$.

	\begin{definition}
		Let $1 \leq q < \infty$, $u_0 \in L^q(\R^d)$ and $T>0$. By a \emph{mild $L^q$-solution} to \eqref{Eq:Heat} on the time interval $[0,T)$ 
		we call a function 
		\begin{equation*}
			u \in C([0,T),L^q(\R^d)) \cap L^p_{{loc}}((0,T)\times \R^d)
		\end{equation*}
		that is a distributional solution to \eqref{Eq:Heat} on $(0,T)\times \R^d$ and for which $u(0,\cdot)=u_0$.
	\end{definition}
	
	In this paper, we are concerned with the problem of existence and uniqueness of mild solutions to \eqref{Eq:Heat}. For a comprehensive overview of previous works, we refer to the book by Quittner-Souplet \cite{QuiSou19}, Section 15 in particular. Here, we give a short and non-inclusive overview of some of the relevant results.  In this context, important role is played by the exponent
	\begin{equation}\label{Def:q_c}
		q_c:= \frac{d(p-1)}{2}.
	\end{equation}
	
	\noindent It is know from the work of Weissler \cite{Wei80} that if $q > q_c$ or $q = q_c > 1$, then for arbitrary $u_0 \in L^q(\R^d)$ there exists $T>0$ and a mild $L^q$-solution to \eqref{Eq:Heat} on $[0,T)$. Furthermore, as shown later by Brezis-Cazenave \cite{BreCaz96}, if one imposes also the boundedness condition 
	\begin{equation*}
		u \in L^\infty_{{loc}}((0,T),L^\infty(\R^d))
	\end{equation*}
	then uniqueness of Weissler's solutions holds as well. For large enough $q$ one can in fact ensure unconditional uniqueness in the whole space $C([0,T),L^q(\R^d))$. Namely, already Weissler \cite{Wei80} showed that this holds when $q > q_c$ and $q \geq p$. Brezis-Cazenave \cite{BreCaz96} then extended his result to $q \geq  q_c$ assuming $q > p$. However, neither of the techniques from the aforementioned two papers apply to $q=q_c=p$. In fact, uniqueness was later shown to fail in this case by Terraneo \cite{Ter02} (this result was already known for the case of the unit ball domain; see Ni-Sacks \cite{NiSac85}). It is worth mentioning also the work of Giga \cite{Gig86} who showed that the condition of space-time integrability
	\begin{equation}\label{Eq:Giga_condition}
		u \in L^{p_1}((0,T),L^{p_2}(\R^d)) \quad \text{where} \quad p_2>q \quad \text{and} \quad  \frac{1}{p_1}=\left(\frac{1}{q}-\frac{1}{p_2}\right)\frac{d}{2},
	\end{equation} 
	ensures uniqueness for $q=q_c>1$. Clearly, \eqref{Eq:Giga_condition} is not satisfied by the solutions constructed by Terraneo \cite{Ter02}.
	
	When $q < q_c$ it is not in general known whether one can associate to every $u_0 \in L^q(\R^d)$ a mild $L^q$-solution to \eqref{Eq:Heat}. For a non-existence result under the assumption of non-negativity of local solutions see \cite{Wei80}. In contrast, it is known in certain cases that there are initial data that lead to multiple mild solutions. The earliest result goes back to Haraux-Weissler \cite{HarWei82} who showed non-uniqueness for all $1 \leq q<q_c$ when the power $p$ is in the range
	\begin{equation}\label{Eq:Haraux_range}
		1 + \frac{2}{d} < p < p_c,
	\end{equation}
	where $p_c$ is the so-called \emph{energy-critical} power
	\begin{equation}
		p_c=1+\frac{4}{d-2}.
	\end{equation}
	They show this by exhibiting for \eqref{Eq:Heat} a non-trivial rapidly decaying expanding self-similar solution, which thereby arises from zero initial datum. Expanding self-similar profiles with rapid decay exist, however, only in the energy-subcritical range \eqref{Eq:Haraux_range}. For the energy-(super)critical case, $p \geq p_c$, the problem of non-uniqueness  is, to our knowledge, open, and is the focus of this paper. We show that non-uniqueness (from non-zero initial data) holds for the following range of powers
	\begin{equation}\label{Eq:Range}
		1+\frac{2}{d} < p < p_{JL},
	\end{equation}
	where $p_{JL}$ stands for the so-called Joseph-Lundgren exponent
	\begin{equation*}
		p_{JL}:=
		\begin{cases}
			\infty & \text{if }~3 \leq d \leq 10,\\
			1+ \displaystyle{\frac{4}{d-4 - 2 \sqrt{d-1}}} & \text{if }~d \geq 11. 
		\end{cases}
	\end{equation*}
	More precisely, we establish the following result.
	\begin{theorem}\label{Thm:main}
		Assume $d \geq 3$ and let $p$ satisfy \eqref{Eq:Range}. Then for any $1 \leq q < q_c$
		there exists a non-trivial initial datum $u_0 \in L^q(\R^d)$ and a time $T>0$ for which there are two different mild $L^q$-solutions to \eqref{Eq:Heat} on $[0,T)$.
	\end{theorem}
	\noindent Some remarks are in order.
	\begin{remark}
		Our proof of Theorem \ref{Thm:main} is an adaptation of the method Jia-\v{S}ver\'{a}k proposed in \cite{JiaSve15} for showing non-uniqueness of Leray-Hopf solutions to (unforced) incompressible 3d Navier-Stokes equations.
		Their approach is based on the assumption that there exists a forward self-similar solution that is linearly unstable in similarity variables. It, however, appears to be very difficult to rigorously verify this assumption; for numerical evidence see \cite{GuiSve23}. In this paper we show that (the analogue of) this assumption is true for equation \eqref{Eq:Heat} for the range of powers $p$ given in \eqref{Eq:Range}. What is more, we show that the restriction \eqref{Eq:Range} is necessary in the class of radial solutions, i.e., there are no linearly unstable radial expanding self-similar solutions to \eqref{Eq:Heat} if $p \geq p_{JL}$. 
		
	\end{remark}
	
	\begin{remark}
		The idea of using unstable expanding self-similar solutions to show non-unique-ness  has been employed in a number of contexts lately, primarily in fluid dynamics; for some of the recent results see, e.g., \cite{DolMes24,CasFarMen24,DeBruAlb24,AlbBruCol22}. These results, however, in contrast to Theorem \ref{Thm:main}, consider equations with non-trivial forcing terms, which serve the purpose of facilitating the construction of unstable similarity profiles.
	\end{remark}

	\begin{remark}
		We note that the initial data that we construct are radial, satisfy
		\begin{equation*}
			u_0(x) = \frac{C}{|x|^{\frac{2}{p-1}}} \quad \text{for some} \quad C>0,
		\end{equation*}
		near zero, and are uniformly bounded otherwise.
		Furthermore, as it will be apparent from the proof, the non-uniqueness mechanism we exhibit persists under perturbations of $u_0$ that are radial and in $L^q \cap L^r$ for certain $r>q_c$. Note, however, that such perturbations do not remove the singular behavior near zero. We contrast this with the case when perturbations are allowed to be chosen from $L^q$, since there are then arbitrarily small deformations that turn such datum into an $L^\infty$ function, which does not fit into our non-uniqueness scheme. This indicates some sort of non-genericity of our non-uniqueness mechanism in $L^q$.
	\end{remark}
	
	\begin{remark}
		Haraux-Weissler asked in \cite{HarWei82} as to whether their non-uniqueness result can be shown for non-zero initial data. Theorem \ref{Thm:main}  gives a positive answer to this question.
	\end{remark}
	
	\subsection{Outline of the proof of the main result} The proof is thematically split into several sections. In Section \ref{Sec:Expanders} we analyze the existence and linear stability of radial expanding self-similar solutions to \eqref{Eq:Heat}. More precisely, we consider solutions of the following form
	\begin{equation}\label{Eq:Expander_intro}
		u(t,x)=\frac{1}{t^{\frac{1}{p-1}}}  U \left( \frac{|x|}{\sqrt{t}} \right), \quad t > 0.
	\end{equation}
	To study stability of \eqref{Eq:Expander_intro}, it is customary to pass to \emph{(radial) similarity variables}
	\begin{equation*}
		\tau :=  \ln t, \quad \rho := \frac{|x|}{\sqrt{t}}.
	\end{equation*}
	By also scaling the solution profile $t^\frac{1}{p-1}u(t,x):=v(\tau,\rho)$,
	from \eqref{Eq:Heat_radial} we arrive at an evolution equation for  $v$
	\begin{equation}\label{Eq:Sim_var_v_intro}
		\partial_\tau v = L_0 v + |v|^{p-1}v.
	\end{equation}
	Here, the linear operator $L_0$ is given by
	\begin{equation*}
		L_0  = \partial_\rho^2 +\frac{d-1}{\rho} \partial_\rho + \frac{1}{2}\rho \partial_\rho + \frac{1}{p-1}.
	\end{equation*}
	Note that the expander profiles $U$ are now static solutions to \eqref{Eq:Sim_var_v_intro}. Then, by linearization around $U$ we get
	\begin{equation}\label{Eq:Sim_var_w_intro}
		\partial_\tau w = L_0w + V w + N(w),
	\end{equation}
	where $	V = p  |U|^{p-1}$
	and $N(w)$ is the nonlinear remainder. We then proceed to construct solutions to \eqref{Eq:Sim_var_w_intro} in the radial intersection Lebesgue spaces
	\begin{equation*}
		L^{q,r}:=L^q_{{rad}} (\R^d) \cap L^r_{{rad}} (\R^d), \quad \| \cdot \|_{L^{q,r}} := \| \cdot \|_{L^q(\R^d)} + \| \cdot \|_{L^r(\R^d)},
	\end{equation*}
	for $q < q_c < r$. For this, we employ semigroup theory. First, we show that $L:=L_0 + V$ generates a strongly continuous semigroup $S(\tau)$ in $L^{q,r}$ with a negative essential growth bound. This, in particular, means that if the  growth bound of $S(\tau)$ is positive, then it is given by the spectral bound $s(L)$. Existence of expanders $U$ for which the operator $L$ has positive yet arbitrarily small spectral bound, is essential for the rest of the paper. In line with this, we follow with the central result of the section, Theorem \ref{Prop:Spectral_properties_L}, which says that the unstable spectrum of $L$ consists of a finite number of real eigenvalues. Furthermore, we show that if \eqref{Eq:Range} holds,
	then \eqref{Eq:Heat} admits a radial expander $\bar{U}$ for which the operator $L$ 
	has at least one positive eigenvalue. What is more, the largest eigenvalue $\bar{\la}$, which is  then equal to $s(L)$, can be made arbitrarily small. We also prove that the range \eqref{Eq:Range} is optimal, in the sense that for $p\geq p_{JL}$ there are no linearly unstable radial expanders.

	In Section \ref{sec: ancient solutions}, we are using  unstable expanders $\bar{U}$ from Section \ref{Sec:Expanders} to construct ancient solutions to 
	\begin{align}\label{eq:self_sim_intro}
		\partial_\tau U=L_0 U+|U|^{p-1}U,\quad x\in \R^d, \quad \tau\in (-\infty,T],
	\end{align}
	of the form $	U=\bar{U}+\psi$,
	where $\psi(\tau) \rightarrow 0$ in $L^{q,r}$ as $\tau \rightarrow -\infty$. The curve $\tau \mapsto \psi(\tau)$, in fact, represents for $\bar{U}$ the unstable manifold that corresponds to the largest unstable eigenvalue $\bar{\la}$. This construction is the content of the main result of this section, Theorem \ref{thm:existence_ancient_solutions}.
	The two solutions $U_1=\bar{U}$ and $U_2=\bar{U} + \psi$ then yield two radial solutions of \eqref{Eq:Heat} that stem from the same singular initial datum 
	\begin{equation*}
		\tilde{u}_0(x)=\frac{\bar{\ell}}{|x|^\frac{2}{p-1}} \quad \quad \text{for some} \quad\bar{\ell}>0.
	\end{equation*}
	This profile, however, fails to belong to $L^q_{rad}$ precisely when $q<q_c$. To enforce integrability, we have to truncate $\tilde{u}_0$, and deform the two radial solution accordingly such that they yield, for positive times, two different mild $L^q$-solutions. This process is the content of the last two sections.
	
	In Section \ref{sec:linear_heat_with_potential} we write the initial datum above as $\tilde{u}_0=u_0+w_0$, for compactly supported and radial $u_0$, which therefore belongs to $L^q_{rad}$, and the cut-off $w_0$, which is in $L^r_{rad}$. Then we analyze the evolution of $w_0$, so as to subtract it from the two radial solutions above to obtain two local $L^q$-solutions $u_1,u_2$ with initial datum $u_0$. The underlying Cauchy problem is given by
	\begin{align}\label{nonlinear PDE_intro}
		\begin{cases}
			\partial_t w=\Delta w+p\lvert \bar{u}\rvert^{p-1}w+f(w),\\
			w(0)=w_0,
		\end{cases}
	\end{align}
	where the potential term is self-similar, i.e.,
	\begin{align*}
		\bar{u}(t,x)=\frac{1}{t^{\frac{1}{p-1}}}\bar{U}\left(\frac{|x|}{\sqrt{t}}\right),
	\end{align*}
	and the forcing is 
	\begin{align*}
		f(w)&=\lvert \bar{u}+u'\rvert^{p-1}(\bar{u}+u') -\lvert \bar{u}+u'-w\rvert^{p-1}(\bar{u}+u'-w)-p\lvert \bar{u}\rvert^{p-1}w;
	\end{align*}
	we denote by $w_1$ a local $L^r$-solution  to \eqref{nonlinear PDE_intro} corresponding to $u'=0$, and $w_2$  corresponding to $u'=\frac{1}{t^{\frac{1}{p-1}}}\psi\big( \ln t,\frac{|x|}{\sqrt{t}}\big)$. Constructing such solutions is, however, far from trivial, due to the fact that the potential term $p|\bar{u}|^{p-1}$ is time dependent and singular at $t=0$. We nevertheless proceed to establish a well-posedness theory for problems of type \eqref{nonlinear PDE_intro}. More precisely, we show that local existence and uniqueness of solutions in $L^r_{rad}$ holds if the largest positive eigenvalue $\bar{\la}$ that corresponds to $\bar{U}$ is small enough, i.e., if
	\begin{equation}\label{Eq:la_bar}
		\bar{\la} < \frac{1}{p-1}-\frac{d}{2r}.
	\end{equation}
	This is the content of Lemma \ref{lem:linear PDE}.
	Note that in lower dimensions, $d \leq 10$, the power $p$ is allowed to be arbitrarily large, which then, due  to \eqref{Eq:la_bar}, forces arbitrarily small choices of $\bar{\la}$. The section ends with Theorem \ref{nonlinear_singular_PDE}, which says that for small enough ancient solutions $\psi$, the Cauchy problem \eqref{nonlinear PDE_intro} admits a local solution in $L^r_{rad}.$ 
	
	In Section \ref{sec:localization}, we use the properties of $w_1$ and $w_2$ constructed in Section \ref{sec:linear_heat_with_potential} to show that $u_1:=\bar{u}-w_1$ and $u_2:=\bar{u}+u'-w_2$ are mild $L^q$-solutions to \eqref{Eq:Heat} with common initial datum $u_0$. Furthermore, by using the fact that the unstable manifold $\psi(\tau)$ from Section \ref{sec: ancient solutions} does not decay to zero (as $\tau \rightarrow -\infty$) faster than the unstable mode corresponding to $\bar{\la}$, we show that $u_1 \neq u_2$. This finishes the proof.
	
	\subsection{Notation and conventions}
	
	Given a closed linear operator $(L,\mathcal{D}(L))$ on a Banach space $X$,  we denote by $\rho(L)$ the resolvent set of $L$, while $\sigma( L):= \mathbb{C} \setminus \rho( L)$ stands for the spectrum of $ L$. By $\mathcal{L}(X)$ we denote the space of bounded linear operators on $X$.  For estimates, we use the convenient asymptotic notation $a \lesssim b$ to say that there is some $C>0$ such that $a \leq Cb$. Sometimes, when it is obvious from the context, we omit explicitly mentioning the parameters on which the choice of the implied constant $C$ does not depend. To emphasize the dependence of $C$ on a parameter, say $p$, we will sometimes write $\lesssim_p$.
	Given a function $f$ that depends on time and space variables, say $t$ and $x$, we will, for convenience, often denote $f(t,\cdot)$ by $f(t)$.
	\section{Forward self-similar solutions}\label{Sec:Expanders}
	\noindent In this section, we will restrict our analysis to radial solutions of \eqref{Eq:Heat}. More precisely, we consider
	\begin{equation}\label{Eq:Heat_radial}
		\begin{cases}
			\partial_t u - \partial_r^2 u -\frac{d-1}{r}\partial_r u = |u|^{p-1}u, \\
			u(0,\cdot)=u_0,
		\end{cases}
	\end{equation}
	where $u=u(t,r)\in \R$, $(t,r)\in [0,\infty) \times [0,\infty)$, $d \geq 3$ and $p>1$.
	\noindent This section is devoted to the study of the existence and stability of \emph{forward self-similar solutions} to \eqref{Eq:Heat_radial}
	\begin{equation}\label{Eq:Expander}
		u(t,r)=\frac{1}{t^{\frac{1}{p-1}}}  U \left( \frac{r}{\sqrt{t}} \right).
	\end{equation}
	Such solution are descriptively also called \emph{expanding self-similar solutions}, or shortly \emph{expanders}. For convenience, we will also refer to profiles $U$ as expanders.
	\subsection{Existence of expanders}
	By plugging the ansatz \eqref{Eq:Expander} in \eqref{Eq:Heat_radial} we arrive at the following nonlinear ODE for the profile $U=U(\rho)$
	\begin{equation}\label{Eq:Expander_eq_radial}
		U'' + \left(\frac{d-1}{\rho}+\frac{\rho}{2}  \right)U' + \frac{1}{p-1} U + |U|^{p-1}U = 0.
	\end{equation}
	Under initial conditions
	\begin{equation}\label{Eq:Init_cond}
		U(0)=\alpha>0 \quad \text{and} \quad U'(0)=0,
	\end{equation}
	\eqref{Eq:Expander_eq_radial} admits a unique (classical) solution near zero.
	The global properties of these solutions have been extensively studied; for some early results see, e.g., \cite{HarWei82,PelTerWei86}, and for a comprehensive overview see \cite[Appendix Ga]{QuiSou19}. Here, we copy a result of Haraux-Weissler from \cite{HarWei82}.
	\begin{proposition}[\cite{HarWei82}]\label{Prop:Har-Wei}
		Let $d \geq 3$ and $p>1$. Then for every $\alpha > 0$ there exists a unique function $U \in C^2[0,\infty)$ that satisfies \eqref{Eq:Init_cond} and solves \eqref{Eq:Expander_eq_radial} on $(0,\infty)$ classically. Moreover, $U$ is bounded and $\lim_{\rho \rightarrow \infty} \rho^{\frac{2}{p-1}}U(\rho)$  exists and is finite.
	\end{proposition}
	
	\noindent To indicate the dependence on $\alpha$, in what follows we will denote the expanders from the theorem above by $U_\alpha$. Furthermore, we denote
	\begin{equation*}
		\ell(\alpha):=\lim_{\rho \rightarrow \infty} \rho^{\frac{2}{p-1}}U_\alpha(\rho).
	\end{equation*}
	Later on, Naito \cite{Nai06} described in more detail the continuity and monotonicity properties of the function $\alpha \mapsto \ell(\alpha)$; see \cite[Theorem 1.1]{Nai06} in particular.

	%

	\subsection{Stability of expanders}

	\noindent Whenever we do not explicitly specify otherwise, we assume in the rest of this section that  $d \geq 3$, $p>1$ and $\alpha>0$. To analyze stability of expanders, it is customary to pass in \eqref{Eq:Heat_radial}  to variables that are adapted to the self-similar nature of \eqref{Eq:Expander}, the so-called \emph{(radial) similarity variables}
	\begin{equation*}
		\tau :=  \ln t, \quad \rho := \frac{r}{\sqrt{t}}.
	\end{equation*}
	By also scaling the dependent variable
	\begin{equation*}
		v(\tau,\rho):=t^\frac{1}{p-1}u(t,r),
	\end{equation*}
	from \eqref{Eq:Heat_radial} we arrive at an evolution equation for  $v$
	\begin{equation}\label{Eq:Sim_var_v}
		\partial_\tau v = L_0 v + |v|^{p-1}v,
	\end{equation}
	where the linear operator $L_0$ is given by
	\begin{equation*}
		L_0  = \partial_\rho^2 +\frac{d-1}{\rho} \partial_\rho + \frac{1}{2}\rho \partial_\rho + \frac{1}{p-1}.
	\end{equation*}
	Note that expander profiles $U_\alpha$ are now static solutions to \eqref{Eq:Sim_var_v}. What we perform below is the (non)linear stability analysis of $U_\alpha$. For this, we first introduce an appropriate functional framework. To begin, for $p \geq 1$ we define the space of radial Lebesgue functions
	\begin{equation*}
		L^p_{rad}(\R^d):= \{ f : [0,\infty) \rightarrow \mathbb{C}~ \vert~ f \text{ is measurable and } \| f \|_{L^p_{rad}(\R^d)} := \| f(|\cdot|) \|_{L^p(\R^d)} < \infty  \}.
	\end{equation*}
	We also need the spaces 
	\begin{align*}
		C_{c,rad}^\infty(\R^d) &:= \{ f : [0,\infty) \rightarrow \mathbb{C} ~ \vert ~ f(|\cdot|) \in C_c^{\infty}(\R^d)  \},\\
		\mathscr{S}_{rad}(\R^d)&:= \{ f : [0,\infty) \rightarrow \mathbb{C} ~ \vert ~ f(|\cdot|) \in \mathscr{S}(\R^d)  \},   
	\end{align*}
	where $C_c^{\infty}(\R^d)$ is the standard test space of smooth and compactly supported functions on $\R^d$ and $\mathscr{S}(\R^d)$ is the space of Schwartz functions on $\R^d$. We note that both $C_{c,rad}^\infty(\R^d)$ and $\mathscr{S}_{rad}(\R^d)$ are dense in $L^{p}_{rad}(\R^d)$.
	For convenience, we will often shortly write $L^p_{{rad}}$, $C_{c,rad}^\infty$ and $\mathscr{S}_{rad}$ for $L^p_{{rad}} (\R^d)$, $C_{c,rad}^\infty(\R^d)$ and $\mathscr{S}_{rad}(\R^d)$ respectively. Now, for $1\leq \eta\leq \gamma$ we define the radial intersection  Lebesgue space\footnote{We decided to use the suggestive notation $L^{\eta,\gamma}$, hoping it will not cause confusion with the more standard usage in the context of Morrey or Lorenz spaces. }
	\begin{equation*}
		L^{\eta,\gamma}:=L^\eta_{{rad}} (\R^d) \cap L^\gamma_{{rad}} (\R^d), \quad \| \cdot \|_{L^{\eta,\gamma}} := \| \cdot \|_{L^\eta(\R^d)} + \mathbf{1}_{(0,\infty)}(\gamma-\eta)\| \cdot \|_{L^\gamma(\R^d)}.
	\end{equation*}
	Note that $L^{\eta,\eta}=L_{rad}^\eta$. Although we define spaces of radial functions on $\R^d$ via their radial profiles, for convenience we will at times interpret them as defined on $\R^d$ via the identification $f(x)=f(|x|)$ for $x \in \R^d$. 
	
	Our aim is to study the flow of \eqref{Eq:Sim_var_v} near $U_\alpha$ in spaces $L^{\eta,\gamma}$. To this end, we substitute $v(\tau,\rho)= U_\alpha(\rho) + w(\tau,\rho)$ into \eqref{Eq:Sim_var_v}. This then leads to an evolution equation for the perturbation $w$
	\begin{equation}\label{Eq:Sim_var_w}
		\partial_\tau w = L_0w + V_\alpha w + N(w),
	\end{equation}
	where
	\begin{equation*}
		V_\alpha = p  |U_\alpha|^{p-1},
	\end{equation*}
	and
	\begin{equation*}
		\quad N_\alpha(w) := n(U_\alpha+w)-n(U_\alpha)-p|U|^{p-1}w \quad \text{for} \quad  n(f)=|f|^{p-1}f. 
	\end{equation*}
	To construct solutions to \eqref{Eq:Sim_var_w} we resort to semigroup theory. First, we study the flow generated by $L_0$, and then we perturbatively treat the linear flow of \eqref{Eq:Sim_var_w}. To this end, we supply $L_0$ with a domain $\mathcal{D}(L_0):=\mathscr{S}_{rad}$. With these preparations at hand, we formulate the first result of this section.
	\begin{proposition}\label{Prop_Semigroup_S0}
		Let $1\leq \eta\leq \gamma$. Then the operator $(L_0,\mathcal{D}(L_0))$ is closable in $L^{\eta,\gamma}$, and its closure (which we also denote by $(L_0,\mathcal{D}(L_0))$) generates a one-parameter strongly continuous semigroup $\left(S_0(\tau)\right)_{\tau \geq 0}\subseteq \mathcal{L}(L^{\eta,\gamma})$, which, for some $M\geq 1$,  satisfies the growth estimate
		\begin{align}\label{estimate_semigroup}
			\lVert S_0(\tau)u_0\rVert_{L^{\eta,\gamma}}\leq Me^{\left(\frac{1}{p-1}-\frac{d}{2\eta}\right)\tau}\lVert u_0\rVert_{L^{\eta,\gamma}},
		\end{align}
		for $u_0\in L^{\eta,\gamma}$ and $\tau\geq 0$. The semigroup $S_0$ is, in fact, given explicitly by the convolution relation in $\R^d$
		\begin{align}\label{semigroup_explicit_expression}
			S_0(\tau)u_0(|\xi|)&= e^{\frac{\tau}{p-1}}G_{\tau}\ast u_0(e^{\frac{\tau}{2}}|\xi|), 
		\end{align}
		where 
		\begin{align*}
			G_{\tau}(\xi)=\left(4\pi \alpha(\tau)\right)^{-d/2}e^{-\frac{\lvert \xi\rvert^2}{4\alpha(\tau)}},\quad \alpha(\tau)=e^\tau-1.
		\end{align*}
		Furthermore, $S_0(\tau)$ satisfies the smoothing estimates
		\begin{align}\label{eq:regularization_semigroup_lplq1}
			\lVert S_0(\tau)u_0 \rVert_{L^{\eta' }_{rad}}&\lesssim\frac{e^{(\frac{1}{p-1}- \frac{d}{2\eta}) \tau}}{\alpha(\tau)^{\frac{d}{2}\left(\frac{1}{\eta}-\frac{1}{\eta'}\right)}}\lVert u_0\rVert_{L^\eta_{rad}},\\
			\label{eq:regularization_semigroup_lplq2} \lVert S_0(\tau)u_0 \rVert_{L^{\eta',\gamma'}}&\lesssim\frac{e^{(\frac{1}{p-1}- \frac{d}{2\eta}) \tau}}{\alpha(\tau)^{\frac{d}{2}\left(\frac{1}{\eta}-\frac{1}{\eta'}\right)}}\lVert u_0\rVert_{L^{\eta,\gamma}},
		\end{align}
		for all $\tau \in (0,2]$, whenever $ \eta\leq \eta',\ \gamma\leq \gamma',\ \frac{1}{\eta}-\frac{1}{\eta'}=\frac{1}{\gamma}-\frac{1}{\gamma'}.$
	\end{proposition}
	\begin{proof}
		The explicit expression \eqref{semigroup_explicit_expression} is simply obtained through self-similar scaling of the solution for the linear heat equation. More explicitly, we observe that $v(t,x)$ solving
		\begin{align*}
			\begin{cases}
				\partial_t v=\Delta v\quad t>0,~x\in\R^d,\\
				v(0,\cdot)=u_0,
			\end{cases}
		\end{align*}
		can be written as $v(t,x)=\frac{1}{(t+1)^{\frac{1}{p-1}}}u \big(\ln(t+1),\frac{x}{\sqrt{t+1}} \big)$, where $u=u(\tau,\xi)$ solves
		\begin{align}\label{system_self_similar_variables}
			\begin{cases}
				\partial_\tau u=L_0 u \quad \tau>0,~\xi\in\R^d, \\
				u(0,\cdot)=u_0.
			\end{cases}
		\end{align}
		Therefore \eqref{semigroup_explicit_expression} follows easily by letting $t=e^{\tau}-1$ and $\ x=\xi \sqrt{e^{\tau}}$ in
		\begin{align}\label{representation_ss_semigroup}
			u(\xi, \tau)& =(t+1)^{\frac{1}{p-1}}v(t,x)\notag \\ &=(t+1)^{\frac{1}{p-1}}\int_{\R^d} \frac{e^{-\frac{\lvert x-y\rvert^2}{4t}}}{(4\pi t)^{\frac{d}{2}}} u_0(y) dy.
		\end{align}
		The representation formula \eqref{representation_ss_semigroup} for solutions of \eqref{system_self_similar_variables}, and the properties of the heat semigroup on $\R^d$ then imply that the family of operators $(S_0(\tau))_{\tau\geq 0}$ is a strongly continuous semigroup on $L^{\eta,\gamma}$; we omit the elementary computations. Let us denote by $(A, \mathcal{D}(A))$ its infinitesimal generator. Again, from the representation formula \eqref{representation_ss_semigroup} and the fact that $\mathscr{S}_{rad}$ is left invariant by the heat semigroup on $\R^d$, it follows that $A|_{\mathscr{S}_{rad}}=L_0|_{\mathscr{S}_{rad}}$ and $\mathscr{S}_{rad}$ is a core for $(A,\mathcal{D}(A))$; see, e.g., \cite[Proposition II.1.7, p.~53]{EngNag00}. The latter implies that $(L_0, \mathcal{D}(L_0))$ is closable in $L^{\eta,\gamma}$, and its closure is the infinitesimal generator of $S_0(\tau).$
		
		Concerning the growth bounds and regularization properties of $S_0(\tau)$, starting from \eqref{semigroup_explicit_expression}, for each $1\leq \eta\leq \theta\leq \infty$ we have by Young's inequality for convolutions that
		\begin{align*}
			\lVert S_0(\tau)u_0\rVert_{L^{\theta}_{rad}}&=e^{\left(\frac{1}{p-1}-\frac{d}{2\theta}\right)\tau} \lVert G_{\tau}\ast u_0\rVert_{L^{\theta}_{rad}}\\ & \leq e^{\left(\frac{1}{p-1}-\frac{d}{2\theta}\right)\tau}  \lVert G_{\tau}\rVert_{L^{\frac{\theta\eta}{\theta\eta+\eta-\theta}}}\lVert u_0\rVert_{L^{\eta}_{rad}}\\ & \lesssim 
			\frac{ e^{\left(\frac{1}{p-1}-\frac{d}{2\theta}\right)\tau} }{\alpha(\tau)^{\frac{d}{2}\left(\frac{1}{\eta}-\frac{1}{\theta}\right)}}
			\lVert u_0\rVert_{L^{\eta}_{rad}},
		\end{align*}
		for all $u_0 \in L^{\eta}$ and $\tau \geq 0$.
		This yields \eqref{eq:regularization_semigroup_lplq1}. Finally, we obtain \eqref{estimate_semigroup} and \eqref{eq:regularization_semigroup_lplq2} from \eqref{eq:regularization_semigroup_lplq1} by separately treating small and large values of $\tau$.
	\end{proof}
	\begin{remark}
		Note that $\frac{1}{p-1}-\frac{d}{2\eta} < 0$ if and only if $\eta < q_c$. The semigroup $S_0$ therefore has exponential decay in $L^{\eta,\gamma}$ exactly when $\eta < q_c$. In other words, the linear flow of \eqref{Eq:Sim_var_v} exhibits exponential decay in spaces $L^{\eta,\gamma}$ precisely for supercritical exponents $\eta$. This is in stark contrast to the linear flow of \eqref{Eq:Heat_radial}, which admits in $L^{\eta,\gamma}$ no exponential decay whatsoever. 
	\end{remark}
	
	Now we proceed with the analysis of the linear flow of \eqref{Eq:Sim_var_w}. The multiplication operator $V_\alpha:L^{\eta,\gamma} \rightarrow L^{\eta,\gamma}$ is obviously bounded, thanks to Proposition \ref{Prop:Har-Wei}, and the operator $L_\alpha:=L_0+V_\alpha$, $\mathcal{D}(L_\alpha):=\mathcal{D}(L_0)$ therefore generates a semigroup in $L^{\eta,\gamma}$.
	We, in fact, have the following result.
	\begin{proposition}\label{Prop:difference_compact}
		Let $1\leq \eta\leq \gamma$. Then for each $\alpha>0$ the operator $L_\alpha:\mathcal{D}(L_\alpha) \subseteq L^{\eta,\gamma} \rightarrow L^{\eta,\gamma}$ generates a one-parameter strongly continuous semigroup $\left(S_{\alpha}(\tau)\right)_{\tau \geq 0}\subseteq \mathcal{L}(L^{\eta,\gamma})$. Furthermore, the difference $S_{\alpha}(\tau)-S_0(\tau)$ is a compact operator on $L^{\eta,\gamma}$ for all $\tau \geq 0$.
	\end{proposition}
	\begin{proof}
		We divide the proof in several steps.
		
		\noindent \emph{Step 1: $L_{\alpha}$ generates a semigroup.} 
		The first part of the proposition follows from the bounded perturbation theorem \cite[Theorem III.1.3, p.~158]{EngNag00}; in particular, we have
		\begin{align*}
			\lVert S_{{\alpha}}(\tau)u_0\rVert_{L^{\eta,\gamma}}\leq M e^{\left(\frac{1}{p-1}-\frac{d}{2\eta}+M\norm{V_{\alpha}}_{\mathcal{L}(L^{\eta,\gamma})}\right)\tau}\norm{u_0}_{L^{\eta,\gamma}},
		\end{align*}
		for $\tau\geq 0$, where $M$ is the constant appearing in \eqref{estimate_semigroup}.\\
		\emph{Step 2: Additional regularization for $S_0$.} 
		We use again the link between the heat semigroup, $P(t)$, and $S_0(\tau)$ given by \eqref{representation_ss_semigroup}. In particular
		$v(t)=P(t)u_0$ is smooth for positive times and it holds
		\begin{align*}
			\norm{v(t)}_{L^{\eta,\gamma}}+\sqrt{t}\norm{\nabla v(t)}_{L^{\eta,\gamma}}\lesssim \norm{u_0}_{L^{\eta,\gamma}}\quad \text{for} \quad t\in (0,e^2-1).
		\end{align*}
		The latter implies that $u(\tau)=S_0(\tau)u_0\in W^{1,\eta}\cap W^{1,\gamma}(\R^d)$ and
		\begin{align}\label{W_1 estimates_S0}
			\sqrt{\tau}\norm{\nabla u(\tau)}_{L^{\eta}\cap L^{\gamma}}\lesssim \norm{u_0}_{L^{\eta,\gamma}}\quad \text{for} \quad  \tau\in (0,2).
		\end{align}\\
		\emph{Step 3: End of the proof.} Let us consider a sequence $\{ u_n \}_{n \in \N} \subseteq L^{\eta,\gamma} $ for which $\norm{u_n}_{L^{\eta,\gamma}}\leq C$ and $u_n\rightharpoonup u$ in $L^{\eta,\gamma}$. By Duhamel's formula and the definition of $S_0$ and $S_{\alpha}$ we have
		\begin{align*}
			S_{0}(\tau)(u_n-u)-S_{\alpha}(\tau)(u_n-u)=-p\int_0^\tau S_{\alpha}(\tau-s)[\lvert \bar{U}_{\alpha}\rvert^{p-1}S_0(s)(u_n-u)]ds.
		\end{align*}
		Due to Propositions \ref{Prop_Semigroup_S0}, and \ref{Prop:Har-Wei}, and Step 1, we have
		\begin{align*}
			\norm{S_{\alpha}(\tau-s)[\lvert \bar{U}_{\alpha}\rvert^{p-1}S_0(s)(u_n-u)]}_{L^{\eta,\gamma}} \leq  2\alpha^{p-1}M^2 C e^{\left(\frac{1}{p-1}-\frac{d}{2\eta}\right)} e^{M\norm{V_{\alpha}}_{\mathcal{L}(L^{\eta,\gamma})}(\tau-s)},
		\end{align*}
		which is uniformly in $n$ integrable in $s$ on $(0,\tau)$.
		In addition, we prove that for each $s\in (0,\tau)$
		\begin{align*}
			\norm{S_{\alpha}(\tau-s)[\lvert \bar{U}_{\alpha}\rvert^{p-1}S_0(s)(u_n-u)]}_{L^{\eta,\gamma}}\rightarrow 0 \quad \text{as} \quad n \rightarrow \infty.
		\end{align*}
		The two relations above imply the required compactness by dominated convergence theorem. 
		According to Step 1, it is enough to prove 
		\begin{align}\label{claim_lebesgue}
			\norm{\lvert \bar{U}_{\alpha}\rvert^{p-1}S_0(s)(u_n-u)}_{L^{\eta,\gamma}}\rightarrow 0.
		\end{align}
		Thanks to to Proposition \ref{Prop:Har-Wei}, there exists $\bar{R}$ large enough such that for each $R\geq \bar{R}$ 
		\begin{align*}
			\lvert \bar{U}_{\alpha}(\xi)\rvert^{p-1}\leq \frac{2\ell(\alpha)^{p-1}}{R^2}\quad \text{if} \quad  \lvert \xi\rvert>R.
		\end{align*}
		Moreover, due to Step 2, $S_0(s)u_n\in W^{1,\eta}\cap W^{1,\gamma}(\R^d)$ and 
		\begin{multline}
			\sqrt{s}\left(\norm{S_0(s)u_n}_{W^{1,\eta}(B_R)}+\norm{S_0(s)u_n}_{W^{1,\gamma}(B_R)}\right) \lesssim \left(\norm{S_0(s)u_n}_{L^{\eta}}+\norm{S_0(s)u_n}_{L^{\gamma}}\right)\lesssim C,
		\end{multline}
		for all $n \in \N$.
		Therefore, since the embedding of $W^{1,\eta}(B_R)\cap W^{1,\gamma}(B_R) $ in $L^{\eta}(B_R)\cap L^{\gamma}(B_R)$ is compact, the weak convergence of $u_n$ to $u$ implies
		\begin{align*}
			\norm{S_0(s)(u_n-u)}_{L^{\eta}(B_R)\cap L^{\gamma}(B_R)}\rightarrow 0\quad \text{for} \quad  s>0.
		\end{align*}
		Combining the two relations above we have
		\begin{align*}
			\limsup_{n\rightarrow+\infty}\norm{\lvert \bar{U}_{\alpha}\rvert^{p-1}S_0(s)(u_n-u)}_{L^{\eta,\gamma}}\lesssim \frac{1}{R^2}+\limsup_{n\rightarrow+\infty}\norm{S_0(s)(u_n-u)}_{L^{\eta}(B_R)\cap L^{\gamma}(B_R)}.
		\end{align*}
		Due to the arbitrariness of $R$, relation \eqref{claim_lebesgue} follows and the proof is complete.
	\end{proof}
	
	Proposition \ref{Prop:difference_compact}  tells us that the essential spectra (and therefore the essential growth bounds) of $S_{\alpha}(\tau)$ and $S_0(\tau)$ are the same for $\tau \geq 0$. In particular, $\sigma(S_{\alpha}(\tau))\setminus \sigma(S_0(\tau))$ consists of isolated eigenvalues of finite algebraic multiplicity. Consequently, in view of the spectral mapping theorem for the point spectrum, to understand growth of $S_{\alpha}$ relative to $S_0$ it suffices to analyze the point spectrum of $L_\alpha$. 
	
	\subsection{Spectral analysis of $L_\alpha$}
	In this section we analyze the spectrum of $L_\alpha : \mathcal{D}(L_\alpha) \subseteq L^{\eta,\gamma} \rightarrow L^{\eta,\gamma}$ for $\alpha\geq 0$. From \eqref{estimate_semigroup} we see that
	\begin{equation}\label{Eq:Spec_L_0}
		\sigma(L_0) \subseteq \big\{ \la \in \mathbb{C} ~\lvert~ \Re \la \leq \tfrac{1}{p-1}-\tfrac{d}{2 \eta} \big\}.
	\end{equation}
	In the sequel, we will consider spaces $L^{\eta,\gamma}$ for which $\eta < q_c$, where $q_c$ is the critical exponent from \eqref{Def:q_c}. This, in particular, implies that $\tfrac{1}{p-1}-\tfrac{d}{2 \eta}<0$ and therefore $\sigma(L_0)$ is strictly contained in the open left half-plane.
	The main result of this section is as follows.
	\begin{theorem}\label{Prop:Spectral_properties_L}
		Assume that
		\begin{equation}
			d \geq 3, \quad p>1+\frac{2}{d}, \quad  \alpha>0, \quad  1 \leq \eta < \frac{d(p-1)}{2}, \quad \text{and} \quad  \eta \leq \gamma.
		\end{equation}
		Then for the operator $L_\alpha : \mathcal{D}(L_\alpha) \subseteq L^{\eta,\gamma} \rightarrow L^{\eta,\gamma}$ the following statements hold.
		\begin{itemize}[leftmargin=8mm]
			\setlength{\itemsep}{2mm}
			\item[1.] The set
			\begin{equation}\label{Eq:unstab_spectr}
				\sigma(L_\alpha) \cap \big\{ \la \in \mathbb{C}~ \lvert~ \Re \la > \tfrac{1}{p-1}-\tfrac{d}{2 \eta} \big\}
			\end{equation}
			consists of finitely many real eigenvalues.
			\item [2.] If $p < p_{JL}$ then for every $\varepsilon>0$ there exists $\alpha>0$ such that $L_\alpha$ admits at least one positive eigenvalue, and furthermore all positive eigenvalues are smaller than $\varepsilon$.

			\item[3.] If $p \geq p_{JL}$ then for every $\alpha>0$ the operator $L_\alpha$ admits no positive eigenvalues.
		\end{itemize}
	\end{theorem}
	\begin{proof}
		According to the spectral mapping theorem for the point spectrum (see, e.g., \cite[Theorem IV.3.7, p.~277]{EngNag00}), Proposition \ref{Prop:difference_compact} implies that the set \eqref{Eq:unstab_spectr} consists of eigenvalues. To prove the rest of the Claim 1, we do the following. Assume that there are $f \in \mathcal D(L_{\alpha})$ and $\la \in \mathbb{C}$ such that
		\begin{equation}\label{Eq:Re_la}
			\Re \la > \frac{1}{p-1}-\frac{d}{2 \eta}
		\end{equation}
		and
		\begin{equation}\label{Eq:Spec_eq}
			L_\alpha f - \la f=0.
		\end{equation}
		Let us recall the operator $L_\alpha$
		\begin{equation*}
			L_\alpha  = \partial_\rho^2 +\frac{d-1}{\rho} \partial_\rho + \frac{1}{2}\rho \partial_\rho + \frac{1}{p-1} + V_\alpha.
		\end{equation*}
		This means that $f$ solves the following ODE
		\begin{equation}\label{Eq:Eigenv_ODE}
			f'' + \left( \frac{d-1}{\rho} +\frac{\rho}{2} \right)f' + \left(\frac{1}{p-1} +V_\alpha -\la \right)f=0.
		\end{equation}
		Linear ODE theory tells us that such $f$ belongs to $C^2(0,\infty)$. We therefore turn to analyzing the asymptotic behavior of $f$ at the endpoints.
		Note that $\rho=0$ is a regular singular point of \eqref{Eq:Eigenv_ODE}, and the corresponding set of Frobenius indices is $\{ 0,2-d \}$.
		This tells us that there are two linearly independent solutions that near $\rho=0$ have the following asymptotics 
		\begin{equation*}
			f_1(\rho) = 1+O(\rho^{2}) \quad \text{and} \quad   
			f_2(\rho) = \rho^{2-d}(1+o(1)).
		\end{equation*}
		The requirement that $f \in \mathcal{D}(L_\alpha)$ rules out the second behavior above, so $f$ must be (a non-zero constant multiple of) the unique solution to \eqref{Eq:Eigenv_ODE} that satisfies
		\begin{equation}\label{Eq:Init_condi}
			f(0)=1 \quad \text{and} \quad f'(0)=0.
		\end{equation}
		The point $\rho=\infty$ is an irregular singular point, and one therefore has to do some hands-on analysis to understand the behavior for large $\rho$. It turns out that \eqref{Eq:Eigenv_ODE} admits two linearly independent solution with the following asymptotics near $\rho=\infty$
		\begin{equation}\label{Eq:Asymp_inf}
			f_1(\rho) = \rho^{2\left(\la-\frac{1}{p-1}\right)}(1+O(\rho^{-2})) \quad \text{and} \quad   
			f_2(\rho) = \rho^{2\left(-\la-\frac{d}{2}+\frac{1}{p-1}\right)}e^{-\frac{\rho^2}{4}}(1+O(\rho^{-2})). 
		\end{equation}
		Now, due to the assumption \eqref{Eq:Re_la}, the requirement that $f \in L^{\eta,\gamma}$ singles out the second behavior above as the only admissible one. In summary, the eigenfunction $f$ is a constant multiple of the solution to \eqref{Eq:Eigenv_ODE} that satisfies \eqref{Eq:Init_condi} and furthermore exponentially decays at $\rho=\infty$ as described in \eqref{Eq:Asymp_inf}. Now that we determined the endpoint asymptotics of $f$, we consider \eqref{Eq:Spec_eq} from a different viewpoint.
		First, define
		\begin{equation*}
			\omega(\rho):=\rho^{d-1}e^{\frac{\rho^2}{4}}.
		\end{equation*}
		Then, note that we can write $L_\alpha$ in the following way
		\begin{equation*}
			L_\alpha  = \frac{1}{\omega}\partial_\rho \left( \omega \partial_\rho  \right) +\frac{1}{p-1} + V_\alpha.
		\end{equation*}
		This implies that $L_\alpha$ has a self-adjoint realization in the weighted $L^2$-space
		\begin{equation*}
			L^2_{\omega}:= \{ f:[0,\infty) \rightarrow \mathbb{C} ~\lvert~f \text{ is measurable and } \int_{[0,\infty)} |f|^2\omega < \infty \},
		\end{equation*}
		with the inner product
		\begin{equation}
			\langle f,g\rangle_{L^2_\omega} := \int_{[0,\infty)}f \bar{g} \omega.
		\end{equation}
		More precisely, the operator $L_\alpha$, when initially defined on $C^\infty_{c,rad}$, is closable in $L^2_\omega$ and its closure (which within this proof we also denote by $L_\alpha$) is a self-adjoint operator. Moreover, $L_\alpha$ has compact resolvent. Consequently, the spectrum of $L_\alpha: \mathcal{D}(L_\alpha) \subseteq L^2_{\omega} \rightarrow L^2_{\omega}$  consists of a discrete set of real simple eigenvalues that can accumulate only at $\infty$. Now, note that due to the exponential growth of the weight function $\omega$, the eigenfunctions of $L_\alpha$ necessarily exhibit the second behavior in \eqref{Eq:Asymp_inf}. Similarly, they have to be regular at $\rho=0$. Based on the first part of the proof, we conclude that the point spectra of $L_\alpha$ in $L^{\eta,\gamma}$ and $L^2_\omega$ match under the assumption \eqref{Eq:Re_la}. This implies Claim 1 of the proposition.
		
		To prove Claims 2 and 3, we employ Sturm-Liouville oscillation theory to $L_\alpha: \mathcal{D}(L_\alpha) \subseteq L^2_{\omega} \rightarrow L^2_{\omega}$. In addition, we will rely on several ODE results from \cite{Nai06}. To count the positive eigenvalues of $L_\alpha$, by Sturm-Liouville, it is enough to count the number of zeros of the unique function $f \in C^2[0,\infty)$ that satisfies
		\begin{equation}\label{Eq:Init_cond_sturm}
			L_\alpha f = 0, \quad f(0)=1, \quad f'(0)=0.
		\end{equation}
		By Theorem 1.1, Lemma 3.3, and Lemma 2.1 in \cite{Nai06} there exists $\alpha^\ast \leq \infty$ such that for all $\alpha < \alpha^\ast$ the solution to \eqref{Eq:Init_cond_sturm} is positive. Claim 3 then follows from the fact that  $\alpha^\ast = \infty$ for $p \geq p_{JL}$, as proven in Corollary 1.2 in \cite{Nai06}.
		It remains to prove Claim 2. First, by Corollary 1.2 in \cite{Nai06} we have that $\alpha < \infty$ for $p < p_{JL}$. Then, by
		Proposition 2.4 and Remark 3.7 in \cite{Nai06}, it follows that for $\alpha=\alpha^\ast$ the solution to \eqref{Eq:Init_cond_sturm} is positive and has exponential decay (due to the alternative \eqref{Eq:Asymp_inf}). This means that $\la=0$ is the only non-negative eigenvalue of $L_{\alpha^\ast}$. Now we prove that $L_{\alpha}$ admits a positive eigenvalue whenever $\alpha > \alpha^{\ast}$. First, for $\alpha > \alpha^\ast$, by Lemma 3.3-(ii) and Lemma 2.1-(i) in \cite{Nai06}, we get that the solution to \eqref{Eq:Init_cond_sturm} has at least one zero, implying that $L_\alpha$ has at least one positive eigenvalue. To prove that $\alpha>\alpha^\ast$ can be chosen such that all positive eigenvalues are arbitrarily small, we do the following. Denote the largest positive eigenvalue by $\lambda_\alpha$. Then
		\begin{align*}
			\la_\alpha = \la_\alpha - 0 
			&= \sup_{f \in C^\infty_{c,rad},\,\| f \|_{L^2_\omega}=1} \langle L_\alpha f,f \rangle_{L^2_\omega} -\sup_{f \in  C^\infty_{c,rad},\| f \|_{L^2_\omega}=1} \langle L_{\alpha^\ast} f,f \rangle_{L^2_\omega} \\
			&\leq \sup_{f \in  C^\infty_{c,rad},\| f \|_{L^2_\omega}=1} \langle L_\alpha f - L_{\alpha^\ast} f,f \rangle_{L^2_\omega} \\
			&= \sup_{f \in  C^\infty_{c,rad},\| f \|_{L^2_\omega}=1} \langle (V_\alpha-V_{\alpha^\ast}) f,f \rangle_{L^2_\omega}\\
			&\leq \| V_\alpha-V_{\alpha^\ast} \|_{L^\infty(0,\infty)}.
		\end{align*}
		Finally, due to the continuity of $\alpha \mapsto V_\alpha$  in $L^\infty(0,\infty)$ we arrive at Claim 2.
	\end{proof}
	
	Now, we turn the spectral information on $L_{\alpha}$ into the growth properties of $S_{\alpha}$.
	\begin{proposition}\label{prop:semigroup}
		Let $1\leq \eta \leq \gamma$ and $\eta  \leq q_c$. Let $\bar{\alpha}>0$ be such that  $L_{\bar{\alpha}}:\mathcal{D}(L_{\bar{\alpha}})\subseteq L^{\eta,\gamma} \rightarrow L^{\eta,\gamma}$ has at least one positive eigenvalue, and denote the largest such eigenvalue by $\la_{\bar{\alpha}}$. Given $\delta>0$ we have that
		\begin{align}\label{eq:decay_semigroup}
			\lVert S_{\bar{\alpha}}(\tau)u_0\rVert_{L^{\eta,\gamma}}\lesssim e^{(\lambda_{\bar{\alpha}} +\delta)\tau}\norm{u_0}_{L^{\eta,\gamma}}
		\end{align}
		for all $u_0 \in L^{\eta,\gamma}$ and $\tau \geq 0$.
		In particular,
		\begin{align}\label{eq:decay_semigroup1}
			\lVert S_{\bar{\alpha}}(\tau) u_0\rVert_{L^{\eta}_{rad}}\lesssim e^{(\lambda_{\bar{\alpha}} +\delta)\tau}\norm{u_0}_{L^{\eta}_{rad}}    
		\end{align}
		for all $u_0 \in L^{\eta}_{rad}$ and $\tau \geq 0$.
	\end{proposition}
	\begin{proof}
		From Propositions \ref{Prop:difference_compact} and \ref{Prop_Semigroup_S0} we get that $\omega_{ess}(S_{\bar{\alpha}}) = \omega_{ess}(S_0) \leq 0$. Then from \eqref{Eq:Spec_L_0} and Theorem \ref{Prop:Spectral_properties_L} we conclude that the spectral bound of $L_{\bar{\alpha}}$ is $s(L_{\bar{\alpha}})=\la_{\bar{\alpha}}$. The claim of the proposition then follows from Proposition \ref{Prop:difference_compact} and the fact that $\omega_0(S_{\bar{\alpha}})=\max \{ \omega_{ess}(S_{\bar{\alpha}}),s(L_{\bar{\alpha}}) \} = s(L_{\bar{\alpha}}) = \la_{\bar{\alpha}}.$ For the standard results from semigroup theory we implicitly invoked in the proof, see, e.g., \cite{EngNag00}, Section IV.2 in particular.
	\end{proof}
	Now we establish smoothing properties of $S_{\bar{\alpha}}$.  To shorten the notation, we denote by $\bar{U}$ the  expander profile that corresponds to $\bar{\alpha}$.
	\begin{proposition}\label{prop:regularization L}
		In addition to the assumptions of Proposition \ref{prop:semigroup}, let $ \eta\leq \eta'$ and $ \gamma\leq \gamma'$ such that $\frac{1}{\eta}-\frac{1}{\eta'}=\frac{1}{\gamma}-\frac{1}{\gamma'}$. Then, given $\delta>0$ we have that
		\begin{align}\label{eq:regularization_L1}
			\lVert S_{\bar{\alpha}}(\tau)u_0 \rVert_{L^{\eta' }_{rad}}&\lesssim\frac{e^{(\lambda_{\bar{\alpha}}+\delta) \tau}}{\tau^{\frac{d}{2}\left(\frac{1}{\eta}-\frac{1}{\eta'}\right)}}\lVert u_0\rVert_{L^\eta_{rad}}
		\end{align}
		for all $u_0 \in L^{\eta}_{rad}$ and $\tau > 0$. Moreover,
		\begin{align}
			\label{eq:regularization_L2} \lVert S_{\bar{\alpha}}(\tau)u_0 \rVert_{L^{\eta',\gamma'}}&\lesssim\frac{e^{(\lambda_{\bar{\alpha}}+\delta) \tau}}{\tau^{\frac{d}{2}\left(\frac{1}{\eta}-\frac{1}{\eta'}\right)}}\lVert u_0\rVert_{L^{\eta,\gamma}} 
		\end{align}
		for all $u_0 \in L^{\eta,\gamma}$ and $\tau > 0$.
	\end{proposition}
	
	\begin{proof}
		Note that by Duhamel formula we have that
		\begin{align*}
			S_{\bar{\alpha}}(\tau)u_0&=S_0(\tau)u_0+p\int_0^\tau S_0(\tau-s)[\lvert \bar{U}\rvert^{p-1}S_{\bar{\alpha}}(s)u_0]ds,
		\end{align*}
		for $u_0 \in L^{\gamma}_{rad}$ and $\tau \geq 0$.
		We now separately treat small and large values of $\tau$.
		First, we assume that $\tau \in (0,2]$. If $\frac{d}{2}\left(\frac{1}{\eta}-\frac{1}{\eta'}\right)<1$ then we have, thanks to \eqref{eq:regularization_semigroup_lplq1} (resp.~\eqref{eq:regularization_semigroup_lplq2}) , \eqref{eq:decay_semigroup1} (resp.~\eqref{eq:decay_semigroup}) and the decay of $\bar{U}$, that
		\begin{align*}
			\lVert S_{\bar{\alpha}}(\tau)u_0\rVert_{L^{\eta'}_{rad}}&\lesssim\frac{1}{\tau^{\frac{d}{2}\left(\frac{1}{\eta}-\frac{1}{\eta'}\right)}}\lVert u_0\rVert_{L^{\eta}_{rad}}+\int_0^\tau \frac{\lVert S_{\bar{\alpha}}(s)u_0\rVert_{L^{\eta}_{rad}}}{(\tau-s)^{\frac{d}{2}\left(\frac{1}{\eta}-\frac{1}{\eta'}\right)}} ds\\ & \lesssim \frac{1}{\tau^{\frac{d}{2}\left(\frac{1}{\eta}-\frac{1}{\eta'}\right)}}\lVert u_0\rVert_{L^{\eta}_{rad}}+\int_0^\tau \frac{e^{(\lambda_{\bar{\alpha}}+\delta)s}}{(\tau-s)^{\frac{d}{2}\left(\frac{1}{\eta}-\frac{1}{\eta'}\right)}}\lVert u_0\rVert_{L^{\eta}_{rad}} ds\\ & \lesssim  \frac{1}{\tau^{\frac{d}{2}\left(\frac{1}{\eta}-\frac{1}{\eta'}\right)}}\lVert u_0\rVert_{L^{\eta}_{rad}}+\lVert u_0\rVert_{L^{\eta}_{rad}}, \\
			( \text{resp. } \lVert S_{\bar{\alpha}}(\tau)u_0\rVert_{L^{\eta',\gamma'}}&\lesssim \frac{1}{\tau^{\frac{d}{2}\left(\frac{1}{\eta}-\frac{1}{\eta'}\right)}}\lVert u_0\rVert_{L^{\eta,\gamma}}+\lVert u_0\rVert_{L^{\eta,\gamma}}\ )
		\end{align*}
		for all $u_0 \in L^{\eta}_{rad}$ and $\tau\in (0,2]$.
		This implies that
		\begin{align}\label{estimate_regularization_small_time1}
			\lVert S_{\bar{\alpha}}(\tau)u_0\rVert_{L^{\eta'}_{rad}}&\lesssim\frac{e^{(\lambda_{\bar{\alpha}}+\delta) \tau}}{\tau^{\frac{d}{2}\left(\frac{1}{\eta}-\frac{1}{\eta'}\right)}}\lVert u_0\rVert_{L^{\eta}_{rad}} \\
			\label{estimate_regularization_small_time2}(\text{resp. }  \lVert S_{\bar{\alpha}}(\tau)u_0\rVert_{L^{\eta',\gamma'}}&\lesssim\frac{e^{(\lambda_{\bar{\alpha}}+\delta) \tau}}{\tau^{\frac{d}{2}\left(\frac{1}{\eta}-\frac{1}{\eta'}\right)}}\lVert u_0\rVert_{L^{\eta,\gamma}} )
		\end{align}
		for all $u_0 \in L^{\eta}_{rad}$ and $\tau\in (0,2]$.
		In case $\frac{d}{2}( \tfrac{1}{\eta}-\tfrac{1}{\eta'} ) \geq 1$, one can successively perform the above steps for a finite number of intermediary values $\eta < \eta_i < \gamma$, thereby obtaining \eqref{estimate_regularization_small_time1}, \eqref{estimate_regularization_small_time2} for $\tau \in (0,2]$ for any choice of $1 \leq \eta \leq \gamma$.

		In order to treat the large values of $\tau$, we use \eqref{estimate_regularization_small_time1} (resp.~\eqref{estimate_regularization_small_time2}) and \eqref{eq:decay_semigroup1} (resp.~\eqref{eq:decay_semigroup}). Indeed, we have that
		\begin{align}\label{estimate_regularization_large_time1}
			\lVert S_{\bar{\alpha}}(\tau)u_0\rVert_{L^{\eta'}_{rad}}&\leq \lVert S_{\bar{\alpha}}(1)\rVert_{L^{\eta}_{rad}\rightarrow L^{\eta'}_{rad}}\lVert S_{\bar{\alpha}}(\tau-1)u_0\rVert_{L^{\eta}_{rad}}\notag\\ & \lesssim e^{(\lambda_{\bar{\alpha}}+\frac{\delta}{2})\tau} \lVert u_0\rVert_{L^{\eta}_{rad}}\\
			\label{estimate_regularization_large_time2} (\text{resp. } \lVert S_{\bar{\alpha}}(\tau)u_0\rVert_{L^{\eta',\gamma'}}&\lesssim e^{(\lambda_{\bar{\alpha}}+\frac{\delta}{2})\tau} \lVert u_0\rVert_{L^{\eta,\gamma}}
		\end{align}
		for all $u_0 \in L^{\eta}_{rad}$ and $\tau > 2$.	Combining \eqref{estimate_regularization_small_time1} (resp.~\eqref{estimate_regularization_small_time2}) and \eqref{estimate_regularization_large_time1} (resp.~\eqref{estimate_regularization_large_time2}) relation \eqref{eq:regularization_L1} (resp.~\eqref{eq:regularization_L2}) follows.
	\end{proof}
	
	\section{Existence of ancient solutions}\label{sec: ancient solutions}
	\noindent Throughout this section we assume that 
	\begin{align}\label{hp: coefficients 1}
		d \geq 3, \quad 1+\frac{2}{d} < p < p_{JL}, \quad 1\leq \hat{q}\leq q_c<\hat{r},\quad \hat{r}\geq \hat{q}p.
	\end{align}
	We are looking for two solutions $U_1, U_2$ of the following PDE
	\begin{align}\label{eq:self_sim}
		\partial_\tau U=L_0 U+|U|^{p-1}U,\quad x\in \R^d, \quad \tau\in (-\infty,T],
	\end{align}
	for some $T \in \R$.
	One of the two solutions, denote it by $U_1$, is independent of time and is given by $\bar{U}$ from the previous section. The other solution, denote it by $U_2$, is of the form 
	\begin{align}\label{eq:form nonuniqueness}
		U_2=\bar{U}+{U}^{lin}+{U}^{per},
	\end{align}
	where ${U}^{lin}$ is the growing mode associated to the maximal unstable eigenvalue $\la_{\bar{\alpha}}$.
	Since we want both $U_1=\bar{U}$ and $U_2=U_1+U^{lin}+U^{per}$ to solve \eqref{eq:self_sim}, we observe that the difference $\psi:=U^{lin}+U^{per}$ has to satisfy the following equation
	\begin{align}\label{eq:perturb}
		\partial_\tau \psi=L_{\bar{\alpha}}\psi+\lvert \bar{U}+\psi\rvert^{p-1}(\bar{U}+\psi)-\lvert \bar{U}\rvert^{p-1}\bar{U}-p\lvert\bar{U}\rvert^{p-1}{\psi}.
	\end{align}
	A function $\psi\in  C((-\infty,\bar{T}], L^{\hat{q},\hat{r}}) $ that solves \eqref{eq:perturb} and for which
	\begin{align*}
		\norm{\psi(\tau)}_{L^{\hat{q},\hat{r}}}\rightarrow 0 \quad \text{ as} \quad \tau\rightarrow -\infty,
	\end{align*}
	will be called an \emph{ancient solution} of equation \eqref{eq:perturb}. For the growing mode $U^{lin}$, we explicitly write
	\begin{align}\label{def:Ulin}
		U^{lin}=e^{\lambda_{\bar{\alpha}}\tau }\bar{U}^{lin},
	\end{align}
	where $\bar{U}^{lin}$ is an eigenfunction of $L_{\bar{\alpha}}$ associated with $\lambda_{\bar{\alpha}}$. In particular
	\begin{align}\label{eq:decay_Ulin}
		\lVert U^{lin}(\tau)\rVert_{L^{\hat{q},\hat{r}}}=e^{\lambda_{\bar{\alpha}} \tau} \lVert\bar{U}^{lin} \rVert_{L^{\hat{q},\hat{r}}},
	\end{align}
	and \begin{align}\label{equation_lin}
		\partial_{\tau}U^{lin}=L_{\bar{\alpha}} U^{lin}.
	\end{align}
	According to \eqref{equation_lin}, from \eqref{eq:perturb} we arrive at an equation for $U^{per}$
	\begin{align}\label{eq:ancient}
		\partial_\tau U^{per}=\,&L_{\bar{\alpha}}U^{per}+\lvert \bar{U}+U^{lin}+U^{per}\rvert^{p-1}(\bar{U}+U^{lin}+U^{per})\notag\\ &-\lvert \bar{U}\rvert^{p-1}\bar{U}-p\lvert\bar{U}\rvert^{p-1}(U^{lin}+U^{per}).
	\end{align}
	The next proposition is the core technical result of this section.
	\begin{proposition}\label{proposition ancient solution}
		Assume \eqref{hp: coefficients 1}. Let $\bar{\alpha}>0$ be such that for the corresponding expander $\bar{U}$ the operator $L_{\bar{\alpha}}: \mathcal{D}(L_{\bar{\alpha}}) \subseteq L^{\hat{q},\hat{r}} \rightarrow  L^{\hat{q},\hat{r}}$ admits a maximal positive eigenvalue $\la_{\bar{\alpha}}$. For every $\delta < \min \{ p-1,1\}\,\la_{\bar{\alpha}}$ there exists $\bar{\eps}=\bar{\eps}(\delta)>0$ such that for all $\eps<\bar{\eps}$ there exists $\bar{T}=\bar{T}(\eps,\delta)<0$ such that for all $T<\bar{T}$  there exists $U^{per}\in C((-\infty,T],L^{\hat{q},\hat{r}})$ that solves \eqref{eq:ancient} and
		\begin{align}\label{decaying_U_per}
			\lVert U^{per}(\tau)\rVert_{L^{\hat{q},\hat{r}}}\leq \eps e^{(\lambda_{\bar{\alpha}} + \delta)\tau}
		\end{align}
		for $\tau \in (-\infty, T]$.
	\end{proposition}
	The proof of Proposition \ref{proposition ancient solution} relies on a fixed point argument. In order to treat the nonlinear part of \eqref{eq:ancient} we need the following elementary result.
	\begin{lemma}\label{lemma_nonlinearity}
		Let $x,y,z\in \R$ and  $p>1$. Then
		\begin{multline}\label{estimate continuity}
			\left\lvert\lvert x+y\rvert^{p-1}(x+y)-\lvert x\rvert^{p-1}x-p\lvert x\rvert^{p-1}y\right\rvert\\ \leq \begin{cases}
				p \rvert y\rvert^p \quad &\text{if } p\leq 2,\\
				\frac{p(p-1)}{2} (1\vee 2^{p-3})\left(\lvert x\rvert^{p-2}\lvert y\rvert^2+\lvert y\rvert^p\right)\quad &\text{if } p>2.
			\end{cases}   
		\end{multline}
		\begin{multline}\label{estimate contraction}
			\left\lvert\lvert x+y\rvert^{p-1}(x+y)-\lvert x+z\rvert^{p-1}(x+z)-p\lvert x\rvert^{p-1}(y-z)\right\rvert\\  \leq \begin{cases}
				p(\lvert y\rvert^{p-1}+\lvert z\rvert^{p-1})\lvert y-z\rvert\quad &\text{if } p\leq 2,\\ 
				p(p-1)(1\vee 3^{p-3})(\lvert y\rvert +\lvert z\rvert)(\lvert x\rvert^{p-2}+\lvert y\rvert^{p-2} +\lvert z\rvert^{p-2})\lvert y-z\rvert & \text{if } p>2.
			\end{cases}
		\end{multline}
	\end{lemma}
	\begin{proof}
		Fix $x,y,z \in \R$. Define $g:\R \rightarrow \R$ by 
		\begin{align*}
			g(t):=\lvert x+ty\rvert^{p-1}(x+ty).
		\end{align*}
		Note that $g$ is differentiable and the derivative $g'(t)=py\lvert x+ty\rvert^{p-1}$ is continuous. We therefore have that
		\begin{align*}
			g(1)-g(0)-g'(0)=g'(\xi)-g'(0) \quad \text{for some} \quad \xi\in (0,1).
		\end{align*}
		Consequently, if $p\leq 2$ then
		\begin{align*}
			\left\lvert\lvert x+y\rvert^{p-1}(x+y)-\lvert x\rvert^{p-1}x-p\lvert x\rvert^{p-1}y\right\rvert\leq p \rvert y\rvert^p.     
		\end{align*}
		If $p>2$ then
		\begin{align*}
			g''(t)=p(p-1)\lvert x+ty\rvert^{p-2}\sgn(x+ty)y^2,
		\end{align*}
		which is continuous, so
		\begin{align*}
			g(1)-g(0)-g'(0)=\frac{1}{2}g''(\xi) \quad \text{for some} \quad \xi\in (0,1).    
		\end{align*}
		Consequently,
		\begin{align*}
			\left\lvert\lvert x+y\rvert^{p-1}(x+y)-\lvert x\rvert^{p-1}x-p\lvert x\rvert^{p-1}y\right\rvert\leq \frac{p(p-1)}{2} (1\vee 2^{p-3})\left(\lvert x\rvert^{p-2}\lvert y\rvert^2+\lvert y\rvert^p\right),
		\end{align*} 
		which concludes the proof of \eqref{estimate continuity}.
		To establish \eqref{estimate contraction} we start by denoting 
		\begin{align*}
			h(t):=\lvert x+ty+(1-t)z\rvert^{p-1}(x+ty+(1-t)z).
		\end{align*}
		Observe that
		\begin{align*}
			h'(t)=p\lvert x+ty+(1-t)z \rvert^{p-1}(y-z).
		\end{align*}
		Therefore, by Lagrange's theorem, we have that
		\begin{align*}
			h(1)-h(0)-p\lvert x\rvert^{p-1}(y-z)=h'(\xi)-p\lvert x\rvert^{p-1}(y-z) \quad \text{for some} \quad \xi\in (0,1).
		\end{align*}
		Consequently, if $p\leq 2$ then  
		\begin{align*}
			\left\lvert\lvert x+y\rvert^{p-1}(x+y)-\lvert x+z\rvert^{p-1}(x+z)-p\lvert x\rvert^{p-1}(y-z)\right\rvert\leq p(\lvert y\rvert^{p-1}+\lvert z\rvert^{p-1})\lvert y-z\rvert.
		\end{align*}
		In case of $p>2$ it remains to estimate 
		\begin{align*}
			p\lvert y-z\rvert \left\lvert\lvert x+\xi y+(1-\xi)z\rvert^{p-1}-\lvert x\rvert^{p-1}\right\rvert
		\end{align*}
		for $\xi\in (0,1)$. Note that the function
		\begin{align*}
			j(s):=\lvert x+s(\xi y+(1-\xi)z)\rvert^{p-1}
		\end{align*}
		is continuously differentiable, and consequently
		\begin{align*}
			j(1)-j(0)= j'(\theta) \quad \text{for some} \quad \theta\in (0,1).
		\end{align*}
		Finally,
		\begin{align*}
			\lvert j(1)-j(0)\rvert \leq (p-1)(1\vee 3^{p-3})(\lvert y\rvert +\lvert z\rvert)(\lvert x\rvert^{p-2}+\lvert y\rvert^{p-2} +\lvert z\rvert^{p-2}),
		\end{align*}
		which concludes the proof of \eqref{estimate contraction}.
	\end{proof}
	We are now ready to prove Proposition \ref{proposition ancient solution}.
	\begin{proof}[Proof of Proposition \ref{proposition ancient solution}]
		Fix a choice \eqref{hp: coefficients 1}. Furthermore, fix $\bar{\alpha}>0$ such that for the corresponding expander $\bar{U}$ the operator $L_{\bar{\alpha}}: \mathcal{D}(L_{\bar{\alpha}}) \subseteq L^{\hat{q},\hat{r}} \rightarrow  L^{\hat{q},\hat{r}}$ admits a maximal positive eigenvalue $\la_{\bar{\alpha}}$. For $\delta>0$ and $T<0$ we define the Banach space
		\begin{align*}
			W_\delta^T:=\{u\in C((-\infty,T],L^{\hat{q},\hat{r}}) ~\lvert~ \| u \|_{W^T_{\delta}}:=\operatorname{sup}_{t\in (-\infty,T]}e^{-(\lambda_{\bar{\alpha}} +\delta)t}\lVert u(t)\rVert_{L^{\hat{q},\hat{r}}}<\infty\}.
		\end{align*}
		Furthermore, given $\eps>0$, we denote by $B_{\eps}(W^T_\delta)$ the closed ball in $W^T_\delta$ with center at $0$ and radius $\eps$. We are looking for $\delta$, $\eps$, and $T$ for which the map 
		$\Gamma: B_{\eps}(W^T_\delta)\rightarrow B_{\eps}(W^T_\delta)$ defined by
		\begin{align*}
			\Gamma(u)(\cdot)=\,&\int_{-\infty}^{\cdot} S_{\bar{\alpha}}(\cdot-\tau)\lvert \bar{U}+U^{lin}(\tau)+u(\tau)\rvert^{p-1}(\bar{U}+U^{lin}(\tau)+u(\tau))d\tau\\ &-\int_{-\infty}^{\cdot}S_{\bar{\alpha}}(\cdot-\tau)\left(\lvert \bar{U}\rvert^{p-1}\bar{U}+p\lvert\bar{U}\rvert^{p-1}(U^{lin}(\tau)+u(\tau))\right) d\tau
		\end{align*}
		is a contraction. We first show that we can arrange that $\Gamma$ indeed maps $B_{\eps}(W^T_\delta)$ to itself. Based on Lemma \ref{lemma_nonlinearity}, let us distinguish two cases, $p \leq 2$ and $p>2$.
		
		Assume $p \leq 2$ and $T<0$. From Proposition \ref{prop:semigroup}, for $\eta=\gamma=\hat{q}$ and Lemma \ref{lemma_nonlinearity} for $x=\bar{U},\ y=U^{lin}+u$, we get that for $t\in (-\infty,T]$
		\begin{align*}
			\lVert \Gamma(u)(t)\rVert_{L^{\hat{q}}_{rad}} &\lesssim\int_{-\infty}^t e^{(\lambda_{\bar{\alpha}}+\delta)(t-\tau)}\left(\lVert U^{lin}(\tau)\rVert_{L^{p\hat{q}}_{rad}}^p+\lVert u(\tau)\rVert_{L^{p\hat{q}}_{rad}}^p\right) d\tau.
		\end{align*}
		Now let us observe that, since $\hat{q}<p\hat{q}\leq \hat{r}$, for a suitable $\theta\in (0,1]$
		\begin{align}\label{interp_ineq_1}
			\lVert f\rVert_{L^{p\hat{q}}_{rad}} & \leq \lVert f\rVert^\theta_{L^{\hat{q}}_{rad}}\lVert f\rVert_{L^{\hat{r}}_{rad}}^{1-\theta} \leq \lVert f\rVert_{L^{\hat{q},\hat{r}}}. 
		\end{align}
		Therefore, thanks to \eqref{eq:decay_Ulin}, we have that
		\begin{align*}
			\lVert U^{lin}(\tau)\rVert_{L^{p\hat{q}}_{rad}}^p \lesssim e^{p\lambda_{\bar{\alpha}} \tau},
		\end{align*}
		and consequently by the definition of $B_{\eps}(W^T_\delta)$
		\begin{align*}
			\lVert u(\tau)\rVert_{L^{p\hat{q}}_{rad}}^p \leq \eps^{p}  e^{p(\lambda_{\bar{\alpha}}+\delta) \tau}.  
		\end{align*}
		In conclusion, denoting by $C=C(\delta)$ the (product of the) hidden constants in the previous steps, we get that
		\begin{align*}
			\lVert \Gamma(u)(t)\rVert_{L^{\hat{q}}_{rad}} \leq&\, Ce^{(\lambda_{\bar{\alpha}}+\delta)t}\int_{-\infty}^t  e^{\tau ((p-1)\lambda_{\bar{\alpha}}-\delta)} d\tau\\ & +C\eps^pe^{(\lambda_{\bar{\alpha}}+\delta)t}\int_{-\infty}^t  e^{\tau (p-1)(\lambda_{\bar{\alpha}}+\delta)} d \tau.
		\end{align*}
		Since $\delta<(p-1)\lambda_{\bar{\alpha}}$
		we get
		\begin{align}\label{estimate_1_pleq2}
			\lVert \Gamma(u)(t)\rVert_{L^{\hat{q}}_{rad}}&\leq   Ce^{((p-1)\lambda_{\bar{\alpha}}-\delta)T} e^{t(\lambda_{\bar{\alpha}}+\delta)}+\eps e^{t(\lambda_{\bar{\alpha}}+\delta)}\left(C\eps^{p-1}e^{t(p-1)(\lambda_{\bar{\alpha}}+\delta)} \right).   \end{align}
		Since $t<0$, if we first choose $\eps$ small enough such that 
		\begin{align*}
			C\eps^{p-1}<\frac{1}{4}    
		\end{align*}
		and consequently $T$ negative enough such that
		\begin{align*}
			Ce^{((p-1)\lambda_{\bar{\alpha}}-\delta)T}\leq \frac{\eps}{4},
		\end{align*}
		we obtain
		\begin{align}\label{first_estimate_map_pleq2}
			\lVert \Gamma(u)(t)\rVert_{L^{\hat{q}}_{rad}}\leq \frac{\eps}{2}e^{t(\lambda_{\bar{\alpha}}+\delta)}.    
		\end{align}
		In order to treat the higher norm $\lVert \Gamma(u)(t)\rVert_{L^{\hat{r}}_{rad}}$, we apply the regularization properties of our semigroup. More precisely, we invoke Proposition \ref{prop:regularization L} for $\eta=\hat{q},\gamma=\frac{\hat{r}}{p},\eta'=\frac{\hat{q}\hat{r}}{\hat{r}-(p-1)\hat{q}}$, and $\gamma'=\hat{r}$. Note that, due to our choices, $s:=d\left(\frac{1}{\gamma}-\frac{1}{\gamma'}\right)=\frac{d(p-1)}{\hat{r}}<2$. Therefore, thanks to Proposition \ref{prop:regularization L}, by arguing as above, we get that
		\begin{align*}
			\lVert \Gamma(u)(t)\rVert_{L^{\hat{r}}_{rad}} &\lesssim \int_{-\infty}^t \frac{e^{(\lambda_{\bar{\alpha}}+\delta)(t-\tau)}}{(t-\tau)^{s/2}}\left(\lVert U^{lin}(\tau)\rVert_{L^{\hat{q},\hat{r}}}^p+\lVert u(\tau)\rVert_{L^{\hat{q},\hat{r}}}^p\right) d\tau,
		\end{align*}
		which implies that
		\begin{equation*}
			\lVert \Gamma(u)(t)\rVert_{L^{\hat{r}}_{rad}} \leq
			Ce^{(\lambda_{\bar{\alpha}}+\delta)t}\int_{-\infty}^t  \frac{e^{\tau ((p-1)\lambda_{\bar{\alpha}}-\delta)}}{(t-\tau)^{s/2}} d\tau +C\eps^pe^{(\lambda_{\bar{\alpha}}+\delta)t}\int_{-\infty}^t  \frac{e^{\tau (p-1)(\lambda_{\bar{\alpha}}+\delta)}}{(t-\tau)^{s/2}} d \tau. 
		\end{equation*}
		The two integrals can be treated similarly to above. Indeed, for the second one, let $\beta$ be such that $\beta s<2$; this and the fact that $\delta<(p-1)\lambda_{\bar{\alpha}}$ then imply
		\begin{align}\label{holder_equation_singular}
			\int_{-\infty}^t  \frac{e^{\tau (p-1)(\lambda_{\bar{\alpha}}+\delta)}}{(t-\tau)^{s/2}} d \tau \leq& \int_{-\infty}^{t-1}  e^{\tau (p-1)(\lambda_{\bar{\alpha}}+\delta)} d \tau\notag\\ & +\left(\int_{-\infty}^t   e^{\beta'\tau (p-1)(\lambda_{\bar{\alpha}}+\delta)} d\tau \right)^{1/\beta'}\left(\int_{t-1}^{t}\frac{1}{(t-\tau)^{\frac{\beta s}{2}}}d\tau\right)^{1/\beta}\notag\\  \lesssim\,& e^{t(p-1)(\lambda_{\bar{\alpha}}+\delta)},
		\end{align}
		where $\beta'$ is the H\"older conjugate of $\beta$. The other integral is analogous. Therefore, up to renaming the constants, we get that
		\begin{align}\label{estimate_2_pleq2}
			\lVert \Gamma(u)(t)\rVert_{L^{\hat{r}}_{rad}} \leq &\, Ce^{((p-1)\lambda_{\bar{\alpha}}-\delta)T} e^{t(\lambda_{\bar{\alpha}}+\delta)}\notag\\ & +\eps e^{t(\lambda_{\bar{\alpha}}+\delta)}\left(C\eps^{p-1}e^{t(p-1)(\lambda_{\bar{\alpha}}+\delta)} \right),  \end{align}
		which is analogous to \eqref{estimate_1_pleq2}. As a consequence, upon possibly choosing a smaller $\eps$ and consequently a more negative $T$, we get 
		\begin{align*}
			\lVert \Gamma(u)(t)\rVert_{L^{\hat{q},\hat{r}}}\leq \eps e^{t(\lambda_{\bar{\alpha}}+\delta)}.    
		\end{align*}
		This implies that $\Gamma$ maps $B_{\eps}(W^T_\delta)$ into itself if $p \leq 2$.
		
		Let us now assume $p > 2$. Arguing as in the case $p \leq 2$ we get that
		\begin{align*}
			\lVert \Gamma(u)(t)\rVert_{L^{\hat{q}}_{rad}}\lesssim & \int_{-\infty}^t e^{(\lambda_{\bar{\alpha}}+\delta)(t-\tau)}\left(\lVert U^{lin}(\tau)\rVert_{L^{p\hat{q}}_{rad}}^p+\lVert u(\tau)\rVert_{L^{p\hat{q}}_{rad}}^p\right) d\tau\\ & +\int_{-\infty}^t e^{(\lambda_{\bar{\alpha}}+\delta)(t-\tau)}\left(\lVert \lvert \bar{U}\rvert^{p-2} U^{lin}(\tau)^2\rVert_{L^{\hat{q}}_{rad}}+\lVert \lvert \bar{U}\rvert^{p-2} u(\tau)^2\rVert_{L^{\hat{q}}_{rad}}\right) d\tau  \\ & =J_1(t)+J_2(t). 
		\end{align*}
		Similarly to above, due to  one can show that $J_1(t)\leq \frac{\eps}{4} e^{(\lambda_{\bar{\alpha}}+\delta)t}$ by choosing $\eps>0$ small enough and then $T$ negative enough. To treat $J_2(t)$ we do the following. Since $\bar{U}\in L^{\theta}(\R^d)$  for $\theta>q_c$, given $f,g\in L^{\hat{q},\hat{r}}$ we have by H\"older's inequality that
		\begin{multline}\label{interp_ineq_2}
			\lVert \lvert \bar{U}\rvert^{p-2} f g\rVert_{L^{\hat{q}}_{rad}}+\lVert \lvert \bar{U}\rvert^{p-2} f g\rVert_{L^{\frac{\hat{r}}{p}}_{rad}}\\ \leq \lVert \bar{U}\rVert_{L^{\hat{r}}_{rad}}^{p-2}\left( \lVert  f\rVert_{L^{\frac{2\hat{q}\hat{r}}{\hat{r}-(p-2)\hat{q}}}_{rad}}\lVert  g\rVert_{L^{\frac{2\hat{q}\hat{r}}{\hat{r}-(p-2)\hat{q}}}_{rad}}+\lVert  f\rVert_{L^{\hat{r}}_{rad}}\lVert  g\rVert_{L^{\hat{r}}_{rad}}\right)\\  \lesssim\lVert f\rVert_{L^{\hat{q},\hat{r}}}\lVert g\rVert_{L^{\hat{q},\hat{r}}}.  
		\end{multline}
		Here we used the fact that $\hat{q}\leq \frac{2\hat{q}\hat{r}}{\hat{r}-(p-2)\hat{q}}\leq \hat{r}$ since $\hat{r}\geq p\hat{q}.$
		Now, we can estimate $J_2(t)$ analogously to $J_1(t)$, obtaining
		\begin{align*}
			J_2(t)& \lesssim \int_{-\infty}^t e^{(\lambda_{\bar{\alpha}}+\delta)(t-\tau)} \left(\lVert U^{lin}(\tau)\rVert_{L^{\hat{q},\hat{r}}}^2+\lVert u(\tau)\rVert_{L^{\hat{q},\hat{r}}}^2\right)d\tau\\ & \lesssim
			\int_{-\infty}^t e^{(\lambda_{\bar{\alpha}}+\delta)(t-\tau)} \left( e^{2\lambda_{\bar{\alpha}}\tau}+\eps^2 e^{2(\lambda_{\bar{\alpha}}+\delta)\tau}\right)d\tau.
		\end{align*}
		Arguing as above, we show that if $\delta<\lambda_{\bar{\alpha}}$ we can find $\eps>0$ small enough and $T$ negative enough such that $J_2(t)\leq \frac{\eps}{4} e^{(\lambda_{\bar{\alpha}}+\delta)t}$.
		Therefore so far we proved
		\begin{align}\label{estimate_1_pgeq2}
			\lVert \Gamma(u)(t)\rVert_{L^{\hat{q}}_{rad}}\leq   \frac{\eps}{2} e^{(\lambda_{\bar{\alpha}}+\delta)t}.  
		\end{align}
		Concerning $\lVert  \Gamma(u)(t)\rVert_{L^{\hat{r}}_{rad}},$ again by Proposition \ref{prop:regularization L}, \eqref{lemma_nonlinearity}, and \eqref{interp_ineq_2} we get that
		\begin{align*}
			\lVert  \Gamma(u)(t)\rVert_{L^{\hat{r}}_{rad}} \lesssim \, &  \int_{-\infty}^t \frac{e^{(\lambda_{\bar{\alpha}}+\delta)(t-\tau)}}{(t-\tau)^{s/2}}\left(\lVert U^{lin}(\tau)\rVert_{L^{\hat{q},\hat{r}}}^p+\lVert u(\tau)\rVert_{L^{\hat{q},\hat{r}}}^p\right) d\tau\\ & +\int_{-\infty}^t \frac{e^{(\lambda_{\bar{\alpha}}+\delta)(t-\tau)}}{(t-\tau)^{s/2}}\left(\lVert \lvert \bar{U}\rvert^{p-2} U^{lin}(\tau)^2\rVert_{L^{\hat{q},\frac{\hat{r}}{p}}}+\lVert \lvert \bar{U}\rvert^{p-2} u(\tau)^2\rVert_{L^{\hat{q},\frac{\hat{r}}{p}}}\right) d\tau \\  \lesssim \, & e^{(\lambda_{\bar{\alpha}}+\delta)t}\int_{-\infty}^t  \frac{e^{\tau ((p-1)\lambda_{\bar{\alpha}}-\delta)}}{(t-\tau)^{s/2}} d\tau+\eps^pe^{(\lambda_{\bar{\alpha}}+\delta)t}\int_{-\infty}^t  \frac{e^{\tau (p-1)(\lambda_{\bar{\alpha}}+\delta)}}{(t-\tau)^{s/2}} d \tau \\ & 
			+e^{(\lambda_{\bar{\alpha}}+\delta)t} \int_{-\infty}^t \frac{e^{\tau (\lambda_{\bar{\alpha}}-\delta)}}{(t-\tau)^{s/2}}d\tau+\eps^2  e^{(\lambda_{\bar{\alpha}}+\delta)t}\int_{-\infty}^t \frac{e^{\tau (\lambda_{\bar{\alpha}}+\delta)}}{(t-\tau)^{s/2}}d\tau,
		\end{align*}
		where in the last inequality we used \eqref{eq:decay_Ulin}, the definition of $W^T_\delta$, and \eqref{interp_ineq_2}.
		Arguing as in \eqref{holder_equation_singular}, we can apply H\"older inequality to the four integrals above, obtaining that, up to some constant $C$,
		\begin{align*}
			\lVert  \Gamma(u)(t)\rVert_{L^{\hat{r}}_{rad}}\leq \,&  \eps e^{t(\lambda_{\bar{\alpha}}+\delta)}\left(C\left(\eps^{p-1}e^{T(p-1)(\lambda_{\bar{\alpha}}+\delta)}+\eps e^{T(\lambda_{\bar{\alpha}}+\delta)}\right) \right)\notag\\ & + C\left(e^{T((p-1)\lambda_{\bar{\alpha}}-\delta)}+ e^{T(\lambda_{\bar{\alpha}}-\delta)}\right)e^{t(\lambda_{\bar{\alpha}}+\delta)}.      
		\end{align*}
		Therefore, we can find $\eps>0$ small enough and consequently $T$ negative enough, possibly smaller then the previous ones such that
		\begin{align*}
			C\left(\eps^{p-1}e^{T(p-1)(\lambda_{\bar{\alpha}}+\delta)}+\eps e^{T(\lambda_{\bar{\alpha}}+\delta)}\right)\leq \frac{1}{4},\quad
			C\left(e^{T((p-1)\lambda_{\bar{\alpha}}-\delta)}+ e^{T(\lambda_{\bar{\alpha}}-\delta)}\right)\leq \frac{\eps}{4}.
		\end{align*}
		As a consequence
		\begin{align*}
			\lVert \Gamma(u)(t)\rVert_{L^{\hat{q},\hat{r}}}\leq \eps e^{t(\lambda_{\bar{\alpha}}+\delta)}.    
		\end{align*}
		This shows that for each $p>1+\frac{2}{d}$ we can find $\delta,\eps,T$ as described in Proposition \ref{proposition ancient solution} such that
		$\Gamma$ maps $B_{\eps}(W^T_\delta)$ into itself.
		
		It remains to show that, by possibly restricting the choice of the parameters, the map $\Gamma$ constructed above is a contraction. This can be done by similar reasoning to the one above, exploiting \eqref{estimate contraction} in place of \eqref{estimate continuity}. Let us start with the case $1+\frac{2}{d}<p\leq 2.$ Let $u, v \in B_{\eps}(W^T_\delta).$ 
		First, we have, according to Proposition \ref{prop:semigroup}, Lemma \ref{lemma_nonlinearity} for $x=\bar{U},\ y=U^{lin}+u,\ z=U^{lin}+v$, H\"older's inequality, and \eqref{interp_ineq_1}, that
		\begin{align*}
			&\lVert \Gamma(u)(t)-\Gamma(v)(t)\rVert_{L^{\hat{q}}_{rad}} \\
			& \lesssim \int_{-\infty}^t e^{(\lambda_{\bar{\alpha}}+\delta)(t-\tau)} \left(\lVert U^{lin}(\tau) \rVert^{p-1}_{L^{\hat{q},\hat{r}}}+\lVert u(\tau) \rVert^{p-1}_{L^{\hat{q},\hat{r}}}+\lVert v(\tau) \rVert^{p-1}_{L^{\hat{q},\hat{r}}}\right)\lVert u(\tau)-v(\tau) \rVert_{L^{\hat{q},\hat{r}}}d\tau.
		\end{align*}
		Since $u,v \in B_{\eps}(W^T_\delta)$, we have that
		\begin{align}\label{estimates contraction1}
			\lVert u(\tau)\rVert_{L^{\hat{q},\hat{r}}}^{p-1}+\lVert v(\tau)\rVert_{L^{\hat{q},\hat{r}}}^{p-1}&\leq 2\eps^{p-1}e^{(p-1)(\lambda_{\bar{\alpha}}+\delta)\tau},\notag\\ \lVert u(\tau)-v(\tau)\rVert_{L^{\hat{q},\hat{r}}}& \leq e^{(\lambda_{\bar{\alpha}}+\delta)\tau}\lVert u-v\rVert_{W^T_\delta}. 
		\end{align}
		Again, by \eqref{eq:decay_Ulin},
		\begin{align*}
			\lVert U^{lin}(\tau)\rVert_{L^{\hat{q},\hat{r}}}^{p-1}= Ce^{(p-1)\lambda_{\bar{\alpha}}\tau}.   
		\end{align*}
		Therefore, since $\delta<(p-1)\lambda_{\bar{\alpha}}$
		\begin{align*}
			\lVert \Gamma(u)(t)-\Gamma(v)(t)\rVert_{L^{\hat{q}}_{rad}}&\lesssim e^{t(p\lambda_{\bar{\alpha}}+\delta)}\left(1+\eps^{p-1}e^{(p-1)\delta t}\right)\lVert u-v\rVert_{W^T_\delta}\\ & \lesssim e^{t(\lambda_{\bar{\alpha}}+\delta)}\lVert u-v\rVert_{W^T_\delta}\left(1+\eps^{p-1}e^{T(p-1)\delta}\right)e^{T(p-1)\lambda_{\bar{\alpha}}}.
		\end{align*}
		By choosing $\eps>0$ small enough and then $T$ negative enough, possibly smaller then the choices above, we get that
		\begin{align}\label{contractionstimate1pleq2}
			\lVert \Gamma(u)(t)-\Gamma(v)(t)\rVert_{L^{\hat{q}}_{rad}}\leq  \frac{1}{4}e^{t(\lambda_{\bar{\alpha}}+\delta)}\lVert u-v\rVert_{W^T_\delta}.   
		\end{align}
		To estimate the higher norm $\lVert \Gamma(u)(t)-\Gamma(v)(t)\rVert_{L^{\hat{r}}_{rad}}$, we do the following. First, according to the regularization properties of the semigroup, i.e., Proposition \ref{prop:regularization L}, by setting $s=\frac{d(p-1)}{\hat{r}}<2$, we get, due to Lemma \ref{lemma_nonlinearity} for $x=\bar{U},\ y=U^{lin}+u,\ z=U^{lin}+v$, H\"older's inequality, equation \eqref{eq:decay_Ulin}, and \eqref{estimates contraction1} that for $t\in (-\infty,T)$
		\begin{equation*}
			\lVert \Gamma(u)(t)-\Gamma(v)(t)\rVert_{L^{\hat{r}}_{rad}}   
			\lesssim  \lVert u-v\rVert_{W^T_\delta} e^{(\lambda_{\bar{\alpha}}+\delta)t} \int_{-\infty}^t \frac{e^{(p-1)\lambda_{\bar{\alpha}}\tau}}{(t-\tau)^{s/2}}(1+\eps^{p-1}e^{(p-1)\delta\tau}) d\tau.
		\end{equation*}
		By arguing as in \eqref{holder_equation_singular}, we get, up to a possibly different constant $C$, that
		\begin{equation*}
			\lVert \Gamma(u)(t)-\Gamma(v)(t)\rVert_{L^{\hat{r}}_{rad}} \leq e^{t(\lambda_{\bar{\alpha}}+\delta)}\lVert u-v\rVert_{W^T_\delta}\left(1+\eps^{p-1}e^{T(p-1)\delta}\right)e^{T(p-1)\lambda_{\bar{\alpha}}}C.
		\end{equation*}
		Choosing $\eps>0$ small enough and then $T$ negative enough, possibly smaller then the previous values, we obtain
		\begin{align}\label{contractionestimate2pleq2}
			\lVert \Gamma(u)(t)-\Gamma(v)(t)\rVert_{L^{\hat{r}}_{rad}}\leq  \frac{1}{4}e^{t(\lambda_{\bar{\alpha}}+\delta)}\lVert u-v\rVert_{W^T_\delta}.
		\end{align}
		Combining \eqref{contractionstimate1pleq2} and \eqref{contractionestimate2pleq2} we get that $\Gamma$ is a contraction on $B_{\eps}(W^T_\delta)$ if $p\in (1+\frac{2}{d},2]$.\\
		Assume now that $p>2$. Arguing as in the case $p \leq 2$ we get by \eqref{interp_ineq_2} that
		\begin{align*}
			\lVert \Gamma(u)(t)&-\Gamma(v)(t)\rVert_{L^{\hat{q}}_{rad}}
			\\ \lesssim\, & \int_{-\infty}^t e^{(\lambda_{\bar{\alpha}}+\delta)(t-\tau)}
			\lVert u(\tau)-v(\tau)\rVert_{L^{\hat{q},\hat{r}}}
			\Big[\big(\lVert U^{lin}(\tau)\rVert_{L^{\hat{q},\hat{r}}}^{p-1}+\lVert u(\tau)\rVert_{L^{\hat{q},\hat{r}}}^{p-1}+\lVert v(\tau)\rVert_{L^{\hat{q},\hat{r}}}^{p-1}\big)\\ & \hspace{6cm} +\big(\lVert U^{lin}(\tau)\rVert_{L^{\hat{q},\hat{r}}}+\lVert u(\tau)\rVert_{L^{\hat{q},\hat{r}}}+\lVert v(\tau)\rVert_{L^{\hat{q},\hat{r}}}\big)\Big] d\tau
			\\  =\,&H_1(t)+H_2(t).
		\end{align*}
		Since now $\delta<\lambda_{\bar{\alpha}}$, the integral $H_1$ can be treated for as the corresponding term in the case $p \leq 2$ in order to find $\eps$ and $T$ such that 
		\begin{align*}
			H_1(t)\leq \frac{1}{8}e^{(\lambda_{\bar{\alpha}}+\delta)t}\lVert u-v\rVert_{W^T_\delta}.
		\end{align*}
		Concerning $H_2(t)$, since $\delta<\lambda_{\bar{\alpha}}$, by \eqref{eq:decay_Ulin} and the definition of $W^T_\delta$, we get 
		\begin{align*}
			H_2(t)&\leq  e^{(\lambda_{\bar{\alpha}}+\delta)t}\lVert u-v\rVert_{W^T_\delta} C\int_{-\infty}^t e^{\lambda_{\bar{\alpha}}\tau}(1+\eps e^{\delta \tau})d\tau\\ & \leq e^{t(\lambda_{\bar{\alpha}}+\delta)}\lVert u-v\rVert_{W^T_\delta} C\left(1+\eps e^{\delta T}\right)e^{\lambda_{\bar{\alpha}}T}.
		\end{align*}
		Possibly reducing $\eps$ and then taking more negative $T$
		we obtain 
		\begin{align*}
			H_2(t)<\frac{1}{8}e^{(\lambda_{\bar{\alpha}}+\delta)t}\lVert u-v\rVert_{W^T_\delta}.
		\end{align*}
		In conclusion, we proved that
		\begin{align}\label{estimate_contraction_1_pgeq2}
			\lVert \Gamma(u)(t)-\Gamma(v)(t)\rVert_{L^{\hat{q}}_{rad}}\leq \frac{1}{4}e^{(\lambda_{\bar{\alpha}}+\delta)t}\lVert u-v\rVert_{W^T_\delta}.
		\end{align}
		Concerning $\lVert \Gamma(u)(t)-\Gamma(v)(t)\rVert_{L^{\hat{r}}_{rad}}$, by the regularization properties of the semigroup, i.e., Proposition \ref{prop:regularization L}, and setting $s=\frac{d(p-1)}{\hat{r}}<2$, we get, according to Lemma \ref{lemma_nonlinearity} for $x=\bar{U},\ y=U^{lin}+u,\ z=U^{lin}+v$, H\"older's inequality, equation \eqref{eq:decay_Ulin}, and \eqref{estimates contraction1} that for $t\in (-\infty,T)$
		\begin{align*}
			\lVert \Gamma(u)(t)&-\Gamma(v)(t)\rVert_{L^{\hat{r}}_{rad}}
			\\ & 
			\lesssim \int_{-\infty}^t \frac{e^{(\lambda_{\bar{\alpha}}+\delta)(t-\tau)}}{(t-\tau)^{s/2}}\lVert u(\tau)-v(\tau)\rVert_{L^{\hat{q},\hat{r}}} \Big[\big(\lVert U^{lin}(\tau)\rVert_{L^{\hat{q},\hat{r}}}^{p-1}+\lVert u(\tau)\rVert_{L^{\hat{q},\hat{r}}}^{p-1}+\lVert v(\tau)\rVert_{L^{\hat{q},\hat{r}}}^{p-1}\big)
			\\ & \hspace{6cm} +\big(\lVert U^{lin}(\tau)\rVert_{L^{\hat{q},\hat{r}}}+\lVert u(\tau)\rVert_{L^{\hat{q},\hat{r}}}+\lVert v(\tau)\rVert_{L^{\hat{q},\hat{r}}}\big)\Big] d\tau
			\\ & =L_1(t)+L_2(t).
		\end{align*}
		The integral $L_1(t)$ can be treated, according to $\delta<\lambda_{\bar{\alpha}}$, as the corresponding term in the analysis of the case $p \leq 2$, so as to find $\eps$ and $T$ such that 
		\begin{align*}
			L_1(t)\leq \frac{1}{8}e^{(\lambda_{\bar{\alpha}}+\delta)t}\lVert u-v\rVert_{W^T_\delta}.
		\end{align*}
		Concerning $L_2(t)$, thanks to the definition of $W^T_\delta$, \eqref{eq:decay_Ulin} and \eqref{holder_equation_singular} we get a term analogous to $H_2(t)$ above. Therefore, up to choosing a smaller $\eps$ and then a more negative $T$ we get 
		\begin{align*}
			L_2(t)\leq \frac{1}{8}e^{(\lambda_{\bar{\alpha}}+\delta)t}\lVert u-v\rVert_{W^T_\delta}.
		\end{align*}
		In conclusion, we proved that
		\begin{align}\label{estimate_contraction_2_pgeq2}
			\lVert \Gamma(u)(t)-\Gamma(v)(t)\rVert_{L^{\hat{r}}_{rad}}\leq \frac{1}{4}e^{(\lambda_{\bar{\alpha}}+\delta)t}\lVert u-v\rVert_{W^T_\delta}.
		\end{align}
		Combining \eqref{estimate_contraction_1_pgeq2} and \eqref{estimate_contraction_2_pgeq2} we conclude that $\Gamma$ is a contraction on $B_{\eps}(W^T_\delta)$ for each $p>1+\frac{2}{d}$, thereby completing the proof.
	\end{proof}
	The different asymptotic behavior of $U^{lin}$ and $U^{per}$ as $\tau\rightarrow-\infty$ implies that $\psi\neq 0$, and consequently $U_1\neq U_2$. As a corollary to Proposition \ref{proposition ancient solution} we now derive the main result of this section; we formulate it in the form of a theorem to be invoked later on.
	\begin{theorem}\label{thm:existence_ancient_solutions}
		Assume \eqref{hp: coefficients 1}. Let $\bar{\alpha}>0$ be such that for the corresponding expander $\bar{U}$ the operator $L_{\bar{\alpha}}: \mathcal{D}(L_{\bar{\alpha}}) \subseteq L^{\hat{q},\hat{r}} \rightarrow  L^{\hat{q},\hat{r}}$ admits a maximal positive eigenvalue $\la_{\bar{\alpha}}$. Then, for every sufficiently small $\eps>0$ there is $T<0$ such that there exists an ancient solution $\psi \in C((\infty,T],L^{\hat{q},\hat{r}})$ to \eqref{eq:perturb} for which 
		\begin{align*}
			\lVert \psi(\tau)\rVert_{L^{\hat{q},\hat{r}}}<\eps,
		\end{align*}
		and
		\begin{align*}
			\lVert \psi(\tau)\rVert_{L^{\hat{r}/p}_{rad}}> e^{\lambda_{\bar{\alpha}}\tau}\frac{\lVert \bar{U}^{lin}\rVert_{L^{\hat{r}/p}_{rad}}}{2},
		\end{align*}
		for  $\tau\in (-\infty,T]$.
	\end{theorem}
	\begin{proof}
		Let us choose $\delta,\eps,T$ such that Proposition \ref{proposition ancient solution} holds. Furthermore, by possibly choosing more negative $T$, we ensure that
		\begin{align}\label{cond 1 barT}
			e^{\lambda_{\bar{\alpha}}T}< \frac{\eps}{2\lVert \bar{U}^{lin}\rVert_{L^{\hat{q},\hat{r}}}},\\
			e^{\delta T}< \frac{\lVert \bar{U}^{lin}\rVert_{L^{\hat{r}/p}_{rad}}}{\eps}. \label{cond 2 barT}
		\end{align}
		Let us now define
		\begin{align*}
			\psi(\tau):=U^{lin}(\tau)+U^{per}(\tau)
		\end{align*}
		for  $\tau\in (-\infty,T]$.
		In particular, we have that $\psi\in C((-\infty,T], L^{\hat{q},\hat{r}})$ and $\psi$ solves \eqref{eq:perturb}. Moreover, by \eqref{eq:decay_Ulin}, \eqref{decaying_U_per}, and the choice of $T$ in \eqref{cond 1 barT}, it follows that
		\begin{align*}
			\lVert \psi(\tau)\rVert_{L^{\hat{q},\hat{r}}}<\eps\quad \text{for} \quad \tau\in(-\infty,T],
		\end{align*}
		and by \eqref{cond 2 barT}, \eqref{decaying_U_per}, and interpolation, we have that
		\begin{align*}
			\lVert U^{per}(\tau)\rVert_{L^{\hat{r}/p}_{rad}}<\frac{\eps}{2}e^{\delta\tau}e^{\lambda_{\bar{\alpha}}\tau}<\frac{\lVert \bar{U}^{lin}\rVert_{L^{\hat{r}/p}_{rad}}}{2}e^{\lambda_{\bar{\alpha}}\tau},  
		\end{align*}
		for  $\tau\in(-\infty,T]$.
		Therefore by \eqref{def:Ulin}
		\begin{align*}
			\lVert \psi(\tau)\rVert_{L^{\hat{r}/p}_{rad}}&\geq \lVert U^{lin}(\tau)\rVert_{L^{\hat{r}/p}_{rad}}-\lVert U^{per}(\tau)\rVert_{L^{\hat{r}/p}_{rad}}   \\ & > \frac{\lVert \bar{U}^{lin}\rVert_{L^{\hat{r}/p}_{rad}}}{2}e^{\lambda_{\bar{\alpha}}\tau},
		\end{align*}
		for  $\tau\in (-\infty,T]$. This completes the proof.
	\end{proof}
	
	\section{Linear heat equation with a self-similar potential}\label{sec:linear_heat_with_potential}
	\noindent The two solutions $U_1,U_2$ that we constructed in the previous section yield two radial solutions of the nonlinear heat equation \eqref{Eq:Heat}
	\begin{align}\label{PDE}
		\partial_t u=\Delta u+\lvert u\rvert^{p-1}u,
	\end{align}
	that, for $q<q_c$, converge locally in $L^q(\R^d)$ to
	\begin{equation}\label{Eq:u_0}
		\tilde{u}_0(x)=\frac{\ell(\bar{\alpha})}{|x|^\frac{2}{p-1}},
	\end{equation}
	as $t \rightarrow 0$.
	This profile, however, fails to belong to $L^q(\R^d)$ precisely when $q<q_c$. To enforce integrability, we truncate $\tilde{u}_0$. However, the part that is cut off has to be such that the alteration of solutions caused by the removal is such that it yields two different solutions that are locally in $L^q(\R^d)$. In the rest of the paper, we show that such construction is possible.
	Throughout, we will assume that
	\begin{align}\label{hp: coefficients 2}
		d \geq 3, \quad 1+\frac{2}{d} < p <p_{JL}, \quad	1\leq q_{a}< q_c<pq_{a}=r.
	\end{align}
	In particular, for the auxiliary parameter $q_a$, we have $q_{a}>\frac{d(p-1)}{2p}$.
	The first step is to derive the equation(s) that the deformed solutions should satisfy. 
	The idea is to write the initial datum $\tilde{u}_0$ as
	$\tilde{u}_0=u_0+w_0$, where $u_0$ is in $L^q_{rad}$ for each $q<q_c$, while $w_0$ only belongs to $L^r_{rad}$. Then, we analyze the equation governing the evolution of $w_0$, so as to subtract it from the two radial solutions above, to obtain two mild $L^q$-solutions with initial datum $u_0$. 
	
	Let $\bar{U}$ and $\psi$ be from Theorem \ref{thm:existence_ancient_solutions} with $\hat{q}=1$ and $\hat{r}=pr$. Then, for $t \in (0,e^T)$ we set
	\begin{align*}
		\tilde{u}_{1}(t,x)&=\frac{1}{t^{\frac{1}{p-1}}}\bar{U}\left(\frac{|x|}{\sqrt{t}}\right),\\ \quad \tilde{u}_{2}(t,x)&=\frac{1}{t^{\frac{1}{p-1}}}\bar{U}\left(\frac{|x|}{\sqrt{t}}\right)+\frac{1}{t^{\frac{1}{p-1}}}\psi\left(\ln t,\frac{|x|}{\sqrt{t}}\right).
	\end{align*}
	Recall that $\bar{U} \in C^2[0,\infty)$ and 
	\begin{align*}
		\bar{U}(\rho)= O\left(\rho^{-\frac{2}{p-1}}\right) \quad \text{as} \quad  \rho \rightarrow \infty.
	\end{align*}
	Theorem \ref{thm:existence_ancient_solutions} furthermore ensures that $\psi\in C((-\infty,T],L^{1,pr})$ and
	\begin{align*}
		\lVert \psi(\tau)\rVert_{L^{1,pr}}<\eps \quad \text{and} \quad \lVert \psi(\tau)\rVert_{L^{r}_{rad}}>e^{\lambda_{\bar{\alpha}}\tau}\frac{\lVert \bar{U}^{lin}\rVert_{L^{r}_{rad}}}{2}
	\end{align*}
	for $\tau\in (-\infty,T]$.
	Now, let us argue formally so as to find the equation to be satisfied by $w$. Let $\tilde{u}$ be of the form $\tilde{u}=\bar{u}+u'$ where both $\tilde{u}$ and $\bar{u}$ solve \eqref{PDE} with the same initial condition $\tilde{u}_0$, while $u$ solves the same equation with the initial condition $u_0$. Therefore, for $w=\tilde{u}-u,$ we have that
	\begin{align*}
		\partial_t w=\,&\partial_t \tilde{u}-\partial_t u\\ =\,&\Delta w+\lvert \tilde{u}\rvert^{p-1}\tilde{u}-\lvert u\rvert^{p-1}u\\    = \,&\Delta w+\lvert \bar{u}+u'\rvert^{p-1}(\bar{u}+u')-\lvert \bar{u}+u'-w\rvert^{p-1}(\bar{u}+u'-w)\pm \lvert \bar{u}\rvert^{p-1}\bar{u}\\  =\,& \Delta w+p\lvert \bar{u}\rvert^{p-1}u'+\left(\lvert \bar{u}+u'\rvert^{p-1}(\bar{u}+u')-\lvert \bar{u}\rvert^{p-1}\bar{u}-p\lvert \bar{u}\rvert^{p-1}u'\right) \\ & -p\lvert \bar{u}\rvert^{p-1}(u'-w)-\left(\lvert \bar{u}+u'-w\rvert^{p-1}(\bar{u}+u'-w)-\lvert \bar{u}\rvert^{p-1}\bar{u}-p\lvert \bar{u}\rvert^{p-1}(u'-w)\right)\\  =\,&\Delta w+p\lvert \bar{u}\rvert^{p-1}w+\left(\lvert \bar{u}+u'\rvert^{p-1}(\bar{u}+u')-\lvert \bar{u}\rvert^{p-1}\bar{u}-p\lvert \bar{u}\rvert^{p-1}u'\right)\\ & -\left(\lvert \bar{u}+u'-w\rvert^{p-1}(\bar{u}+u'-w)-\lvert \bar{u}\rvert^{p-1}\bar{u}-p\lvert \bar{u}\rvert^{p-1}(u'-w)\right).
	\end{align*}
	In conclusion, we have the Cauchy problem
	\begin{align}\label{nonlinear PDE}
		\begin{cases}
			\partial_t w=\Delta w+p\lvert \bar{u}\rvert^{p-1}w+f(w),\\
			w(0)=w_0,
		\end{cases}
	\end{align}
	where
	\begin{align}\label{Def_f(w)}
		f(w)&=\lvert \bar{u}+u'\rvert^{p-1}(\bar{u}+u') -\lvert \bar{u}+u'-w\rvert^{p-1}(\bar{u}+u'-w)-p\lvert \bar{u}\rvert^{p-1}w.
	\end{align}
	In view of the discussion above, we assume that
	\begin{align}\label{Eq:u_bar}
		\bar{u}(t,x)=\frac{1}{t^{\frac{1}{p-1}}}\bar{U}\left(\frac{|x|}{\sqrt{t}}\right). 
	\end{align}
	Then, in case of $\tilde{u}_{1}$ (resp.~$\tilde{u}_2$) we have that $u'=0$ (resp.~$u'=\frac{1}{t^{\frac{1}{p-1}}}\psi\big( \ln t,\frac{|x|}{\sqrt{t}}\big)$). 
	
	Now, under the above assumptions on the time dependent potential $\bar{V}:=p|\bar{u}|^{p-1}$ and the forcing term $f$, we aim at constructing local solutions to \eqref{nonlinear PDE} in $L^r_{rad}(\R^d)$. This is, however, not straightforward due to the fact that $\bar{V}$ is singular at $t=0$ and does not belong to $L^r_{rad}(\R^d)$. Local existence and uniqueness of solutions turns out to hold if the maximal positive eigenvalue $\la_{\bar{\alpha}}$ of $\bar{U}$ is small enough (see \eqref{Item instabilty} below). This well-posedness result is one of the main points of the following technical lemma, which, at the same time, sets the stage for the main result of this section.
	\begin{lemma}\label{lem:linear PDE}
		Assume \eqref{hp: coefficients 2}. Let $\bar{\alpha}>0$ be such that for the corresponding expander $\bar{U}$ the operator $L_{\bar{\alpha}}: \mathcal{D}(L_{\bar{\alpha}}) \subseteq L^{q_a,r} \rightarrow  L^{q_a,r}$ admits a maximal positive eigenvalue $\la_{\bar{\alpha}}$ for which
		\begin{align}
			\label{Item instabilty} {\lambda}_{\bar{\alpha}}<\frac{1}{p-1}-\frac{d}{2r}.  
		\end{align}
		Assume further that $w_0\in L^r_{rad}(\R^d)$ and $f:(0,1)\times \R^d \rightarrow \R$ is such that
		\begin{align*}
			M:=\sup_{t\in (0,1)} \left(\lVert f(t)\rVert_{L^r_{rad}}t+\lVert f(t)\rVert_{L^{q_{a}}_{rad}}t^{1+\frac{d}{2r}-\frac{d}{2q_a}}\right) <\infty,
		\end{align*}
		and
		\begin{align*}
			\operatorname{lim}_{t\rightarrow 0} \left( \lVert f(t)\rVert_{L^r_{rad}}t+\lVert f(t)\rVert_{L^{q_{a}}_{rad}}t^{1+\frac{d}{2r}-\frac{d}{2q_a}} \right)=0.
		\end{align*}
		Then there exists a unique $w\in C([0,1],L^r_{rad}(\R^d))$ that solves the Cauchy problem 
		\begin{align}\label{linear pde}
			\begin{cases}
				\partial_t w=\Delta w+p\lvert \bar{u}\rvert^{p-1}w+f,\\
				w(0)=w_0, 
			\end{cases} 
		\end{align}
		where $\bar{u}$ given by \eqref{Eq:u_bar}.
		Furthermore, $w$ satisfies
		\begin{align}\label{en_est_lin_1}
			\lVert w\rVert_{ L^{\infty}((0,1),L^r_{rad})}+\operatorname{sup}_{t\in (0,1)}t^{\frac{d}{2r}\left(\frac{p-1}{p}\right)}\lVert w(t)\rVert_{ L^{pr}_{rad}}\lesssim M+\lVert w_0\rVert_{L^r_{rad}},
		\end{align}
		\begin{align}\label{en_est_lin_2}
			\operatorname{lim}_{t\rightarrow 0} t^{\frac{d}{2r}\left(\frac{p-1}{p}\right)}&\lVert w(t)\rVert_{ L^{pr}_{rad}}=0.
		\end{align}
	\end{lemma}
	\begin{proof}
		Let us start with the case $w_0\in C^{\infty}_c(\R^d),\ f\in C^{\infty}_c((0,1)\times \R^d)$. Let us look for a solution $w$ of the form
		\begin{align*}
			w(t,x)=P(t)w_0+\phi\left(\ln t,\frac{x}{\sqrt{t}}\right),
		\end{align*}
		$P(t)$ being the heat semigroup on $\R^d$.
		Let us define  $h:(0,\infty)\times \R^d \rightarrow \R$ and $g,b:(-\infty,0)\times \R^d \rightarrow \R$ by 
		\begin{align*}
			f(t,x)=\frac{1}{t}g\left(\ln t,\frac{x}{\sqrt{t}}\right),\quad h(t,x)=P(t)w_0,\quad h(t,x)=b\left(\ln t,\frac{x}{\sqrt{t}}\right).
		\end{align*}
		With this notation in mind, we find that $\phi$ solves
		\begin{align*}
			&\partial_\tau \phi=(L_{\bar{\alpha}}-\tfrac{1}{p-1} )\phi+p\lvert \bar{U}\rvert^{p-1}b+g,\quad (\tau,y)\in (-\infty,0)\times \R^d.
		\end{align*}
		Formally, a solution to the equation above is
		\begin{align}\label{Eq:Intergal}
			\phi(\tau)&=\int_{-\infty}^\tau e^{-\frac{1}{p-1}(\tau-s)}S_{\bar{\alpha}}(\tau-s)[p\lvert \bar{U}\rvert^{p-1}b+g ](s)ds.
		\end{align}
		In what follows, we show that the integral above, in fact, converges in $L^{q_a,r}$. By assumptions, we have that
		\begin{align}\label{forcing_scaling_ss}
			\lVert g(\tau)\rVert_{L^{q_a,r}}\leq e^{-\frac{\tau d}{2r}}M.
		\end{align}
		Secondly, by scaling arguments we have for each $\gamma\in [1,+\infty]$ 
		\begin{align*}
			\lVert b(\tau)\rVert_{L^\gamma_{rad}}=e^{-\frac{ \tau d}{2\gamma}}\lVert h(e^\tau)\rVert_{L^{\gamma}_{rad}}.
		\end{align*}
		This implies, by the contraction properties of the heat semigroup, that
		\begin{align*}
			\lVert b(\tau)\rVert_{L^r_{rad}}\leq e^{-\frac{ \tau d}{2r}}\lVert w_0\rVert_{L^{r}_{rad}}.
		\end{align*}
		Therefore
		\begin{align}\label{initial_condition_scaling_1}
			\lVert \lvert\bar{U}\rvert^{p-1}b(\tau)\rVert_{L^r_{rad}}& \leq \lVert \bar{U}\rVert_{L^{\infty}}^{p-1}\lVert w_0\rVert_{L^{r}_{rad}}e^{-\frac{\tau d}{2r}}. 
		\end{align}
		In order to estimate $\lVert \lvert\bar{U}\rvert^{p-1}b(\tau)\rVert_{L^{q_a}_{rad}}$ we use the fact that $r=p{q_a}$. By H\"older's inequality
		\begin{align}\label{initial_condition_scaling_2}
			\lVert \lvert\bar{U}\rvert^{p-1}b(\tau)\rVert_{L^{q_a}_{rad}}&  \leq \lVert \bar{U}\rVert_{L^{r}_{rad}}^{p-1}\lVert w_0\rVert_{L^{r}_{rad}}e^{-\frac{ \tau d}{2r}}.
		\end{align}
		In conclusion, thanks to Proposition \ref{prop:semigroup},
		we get
		\begin{align}\label{semigroup inequality}
			\lVert \phi(\tau)\rVert_{L^{q_a,r}}\lesssim (M+\lVert w_0\rVert_{L^r_{rad}})\int_{-\infty}^\tau e^{(\lambda_{\bar{\alpha}}+\delta-\frac{1}{p-1})(\tau-s)-\frac{s d}{2r}} ds.
		\end{align}
		According to \eqref{Item instabilty} we can take $\delta$ small enough such that $\lambda_{\bar{\alpha}}+\delta-\frac{1}{p-1} <- \frac{d}{2r}$. Then \eqref{semigroup inequality} implies that the integral in \eqref{Eq:Intergal} converges and
		\begin{align*}
			\lVert \phi(\tau)\rVert_{L^{q_a,r}}\lesssim (M+\lVert w_0\rVert_{L^r_{rad}})e^{-\frac{\tau d}{2r}}.
		\end{align*}
		Consequently
		\begin{align}\label{estimate_1_linear_pde_energy}
			\operatorname{sup}_{t\in (0,1)}\lVert w(t)-h(t)\rVert_{L^r_{rad}}&=\operatorname{sup}_{t\in (0,1)}t^{\frac{d}{2r}}\lVert \phi(\ln t)\rVert_{L^r_{rad}}\notag\\ & \lesssim (M+\lVert w_0\rVert_{L^r_{rad}}).
		\end{align}
		Relation \eqref{estimate_1_linear_pde_energy} and the contraction property of the heat semigroup imply
		\begin{align}\label{estimate_2_linear_pde_energy}
			\operatorname{sup}_{t\in (0,1)}\lVert w(t)\rVert_{L^r_{rad}}&\leq   \operatorname{sup}_{t\in (0,1)}\lVert w(t)-h(t)\rVert_{L^r_{rad}}+  \operatorname{sup}_{t\in (0,1)}\lVert h(t)\rVert_{L^r_{rad}}\notag\\ &\lesssim M+\lVert w_0\rVert_{L^r_{rad}}. 
		\end{align}
		To show \eqref{en_est_lin_1} and \eqref{en_est_lin_2} we now consider $\lVert \phi(\tau)\rVert_{L^{pr}_{rad}}$. By Proposition \ref{prop:regularization L}, setting $\theta=\frac{d(p-1)}{pr}<2$ we get, thanks to \eqref{forcing_scaling_ss}, \eqref{initial_condition_scaling_1}, \eqref{initial_condition_scaling_2}, that
		\begin{align*}
			\lVert \phi(\tau)\rVert_{L^{pr}_{rad}}&\lesssim(M+\lVert w_0\rVert_{L^r_{rad}})\int_{-\infty}^\tau \frac{e^{(\lambda_{\bar{\alpha}}+\delta-\frac{1}{p-1})(\tau-s)-\frac{s d}{2r}} }{(\tau-s)^{\theta/2}}ds. 
		\end{align*}
		Arguing by H\"older's inequality, as in \eqref{holder_equation_singular}, the latter implies
		\begin{align*}
			\lVert \phi(\tau)\rVert_{L^{pr}_{rad}}\lesssim (M+\lVert w_0\rVert_{L^r_{rad}})e^{-\frac{\tau d}{2r}}.
		\end{align*}
		Therefore
		\begin{align}\label{estimate_3_linear_pde_energy}
			\lVert w(t)-h(t)\rVert_{L^{pr}_{rad}}&=t^{\frac{d}{2pr}}\lVert \phi(\ln t)\rVert_{L^{pr}_{rad}}\notag\\ & \lesssim (M+\lVert w_0\rVert_{L^r_{rad}})t^{\frac{d}{2r}\left(\frac{1-p}{p}\right)}.
		\end{align}
		Relation \eqref{estimate_3_linear_pde_energy} and the ultracontractivity property of the heat semigroup imply
		\begin{align}\label{estimate_4_linear_pde_energy}
			\operatorname{sup}_{t\in (0,1)} t^{\frac{d}{2r}\left(\frac{p-1}{p}\right)}\lVert w(t)\rVert_{L^{pr}_{rad}}&\leq   \operatorname{sup}_{t\in (0,1)}t^{\frac{d}{2r}\left(\frac{p-1}{p}\right)}\lVert w(t)-h(t)\rVert_{L^{pr}_{rad}}\notag\\ &+  \operatorname{sup}_{t\in (0,1)}t^{\frac{d}{2r}\left(\frac{p-1}{p}\right)}\lVert h(t)\rVert_{L^{pr}_{rad}}\notag\\ &\lesssim M+\lVert w_0\rVert_{L^r_{rad}}.    
		\end{align}
		The computations above do not use the additional regularity of $w_0$ and $f$ and imply the validity of \eqref{en_est_lin_1} combining \eqref{estimate_2_linear_pde_energy} and \eqref{estimate_4_linear_pde_energy}. In case of the additional regularity, we have
		\begin{align}\label{estimtes_}
			\operatorname{sup}_{t\in [0,1]} \lVert f(t)\rVert_{L^{q_a,r}}\lesssim 1,
		\end{align}
		therefore
		\begin{align*}
			\lVert g(\tau)\rVert_{L^{{q_a},r}}\lesssim e^{\tau\left(1-\frac{d}{2q_a}\right)} 
		\end{align*}
		and $1-\frac{d}{2q_a}>-\frac{d}{2r}$ thanks to \eqref{hp: coefficients 2}.
		For the other term, 
		we can apply H\"older's inequality to obtain
		\begin{align*}
			\lVert \lvert\bar{U}\rvert^{p-1}b(\tau)\rVert_{L^r_{rad}}\leq C(\bar{U})\lVert w_0\rVert_{L^{\Tilde{r}}_{rad}}e^{-\frac{\tau d}{2\Tilde{r}}}
		\end{align*}
		with  $-\frac{d}{2\Tilde{r}}>-\frac{d}{2r}$. By denoting 
		\begin{align}\label{def:delta}
			\vartheta:=\min\{ 1-\frac{d}{2q_a}, -\frac{d}{2\Tilde{r}} \} >-\frac{d}{2r},
		\end{align}
		we can perform the same computations as above to obtain
		\begin{align*}
			\lVert \phi(\tau)\rVert_{L^{pr}_{rad}}+\lVert \phi(\tau)\rVert_{L^{q_a,r}}\lesssim (M+\lVert w_0\rVert_{L^r_{rad}})e^{\vartheta \tau}.
		\end{align*}
		In this case
		\begin{align*}
			\lVert w(t)-h(t)\rVert_{L^r_{rad}}&=t^{\frac{d}{2r}}\lVert \phi(\ln t)\rVert_{L^r_{rad}}\\ & \lesssim t^{\frac{d}{2r}+\vartheta}\rightarrow 0, 
		\end{align*}
		as $t \rightarrow 0$, due to \eqref{def:delta}. Similarly 
		\begin{align*}
			t^{\frac{d}{2r}\left(\frac{p-1}{p}\right)}\lVert w(t)\rVert_{L^{pr}_{rad}}&\leq   t^{\frac{d}{2r}\left(\frac{p-1}{p}\right)}\lVert w(t)-h(t)\rVert_{L^{pr}_{rad}}+  t^{\frac{d}{2r}\left(\frac{p-1}{p}\right)}\lVert h(t)\rVert_{L^{pr}_{rad}}\notag\\ &\lesssim t^{\frac{d}{2r}+\vartheta}\rightarrow 0,
		\end{align*}
		as $t \rightarrow 0$, due to \eqref{def:delta}, and the fact that $w_0\in C^{\infty}_c(\R^d)$. The continuity for positive times follows from the fact that the forcing term $f$ is no more singular for $t>0$.
		This completes the proof of the existence of solution in case of smooth data, the general case follows by approximation.
		
		Concerning the uniqueness, let us assume $w_0,f\equiv 0$. Then any solution $w$ satisfies the mild formula
		\begin{align*}
			w(t)=p\int_0^t P(t-t')[\lvert \bar{u}\rvert^{p-1}w](t') dt'.
		\end{align*}
		By the contraction properties of the heat semigroup and H\"older's inequality, we have
		\begin{align*}
			\lVert w(t)\rVert_{L^{q_a}_{rad}}&\leq p \int_0^t \lVert \lvert \bar{u}(t')\rvert^{p-1}w(t')\rVert_{L^{q_a}_{rad}} dt' \\ &\lesssim_p \int_0^t \lVert \bar{u}(t')\rVert_{L^{r}_{rad}}^{p-1}\lVert w(t')\rVert_{L^r_{rad}} dt'.
		\end{align*}
		Since $\lVert w(t')\rVert_{L^r_{rad}}$ is bounded by the previous existence result and 
		\begin{align*}
			\lVert \bar{u}(t')\rVert_{L^{r}_{rad}}=\frac{1}{t'^{\frac{1}{p-1}-\frac{d}{2r}}}\lVert \bar{U}\rVert_{L^{r}_{rad}},
		\end{align*}
		we obtain
		\begin{align}\label{estimate L^q}
			\lVert w(t)\rVert_{L^{q_a}_{rad}}&\lesssim t^{\frac{d(p-1)}{2r}}.  
		\end{align}
		Recalling that $w(t,x)=\phi\left(\ln t,\frac{x}{\sqrt{t}}\right)$, due to $w_0,f\equiv 0$, we get
		\begin{align*}
			\phi(\tau)=e^{-\frac{1}{p-1}(\tau+s)}S_{\bar{\alpha}}(\tau+s)\phi(-s),
		\end{align*}
		for $\tau,s \in \R$.
		Therefore, thanks to Proposition \ref{prop:semigroup}, estimate \eqref{estimate L^q} and the uniform bound on the $L^r$ norm of $w$, we have that
		\begin{align*}
			\lVert \phi(\tau)\rVert_{L^{q_a,r}}&\lesssim e^{(\lambda_{\bar{\alpha}}+\delta-\frac{1}{p-1})(\tau+{s})}\lVert \phi(-s)\rVert_{L^{q_a,r}}\\ & \lesssim_\tau e^{(\lambda_{\bar{\alpha}}+\delta+\frac{d}{2r}-\frac{1}{p-1}){s}}\rightarrow 0 \quad \text{as} \quad s\rightarrow +\infty,
		\end{align*}
		due to the choice of $\delta$. Thus $\phi\equiv 0$ and consequently also $w$.
	\end{proof}
	Before we formulate the main result of this section, we introduce an auxiliary function space. Under assumptions \eqref{hp: coefficients 2}, for $T'>0$ we define the Banach space
	\begin{align*}
		Z^{T'}:=\{w\in L^{\infty}((0,T'),&L^r_{rad}(\R^d)) ~\lvert \\
		& \sup_{t\in (0,T')}t^{\frac{d}{2r}\left(\frac{p-1}{p}\right)}\lVert w(t)\rVert_{ L^{pr}_{rad}}<\infty,~ \lim_{t\rightarrow 0}t^{\frac{d}{2r}\left(\frac{p-1}{p}\right)}\lVert w(t)\rVert_{ L^{pr}_{rad}}=0\},
	\end{align*}
	equipped with its natural norm
	\begin{align*}
		\| w \|_{Z^{T'}}:=\sup_{t\in (0,T')}\lVert w(t)\rVert_{L^r_{rad}}+ \sup_{t\in (0,T')}t^{\frac{d}{2r}\left(\frac{p-1}{p}\right)}\lVert w(t)\rVert_{ L^{pr}_{rad}}.
	\end{align*}

	\begin{theorem}\label{nonlinear_singular_PDE}
		Assume \eqref{hp: coefficients 2}.  Let $\bar{\alpha}>0$ be such that for the corresponding expander $\bar{U}$ the operator $L_{\bar{\alpha}}: \mathcal{D}(L_{\bar{\alpha}}) \subseteq L^{q_a,r} \rightarrow  L^{q_a,r}$ admits a maximal positive eigenvalue $\la_{\bar{\alpha}}$ that satisfies \eqref{Item instabilty}. Assume further that $w_0\in L^r_\text{rad}(\R^d)$ and $u':(0,1)\times \R^d \rightarrow \R$ is such that it can be written as $u'(t,x)=\frac{1}{t^{\frac{1}{p-1}}}U'\left(\ln t,\frac{|x|}{\sqrt{t}} \right)$, where
		\begin{align*}
			\sup_{\tau \in (-\infty,T]} \lVert U'(\tau,\cdot)\rVert_{L^{1,pr}}<\infty
		\end{align*}
		for some $T<0$. Then, whenever the above displayed quantity is sufficiently small, there is $0<T'<e^T$ such that there exists a unique in $Z^{T'}$ solution $w \in C([0,T'],L^r_{rad}(\R^d))$ to the Cauchy problem \eqref{nonlinear PDE}-\eqref{Def_f(w)}-\eqref{Eq:u_bar}.
		Moreover 
		\begin{align}\label{energy_estimate_nonlin_1}
			\lVert w\rVert_{ L^{\infty}((0,T'),L^r_{rad})}+\operatorname{sup}_{t\in (0,T')}t^{\frac{d}{2r}\left(\frac{p-1}{p}\right)}\lVert w(t)\rVert_{ L^{pr}_{rad}}&\lesssim \lVert w_0\rVert_{L^r_{rad}},
		\end{align}
		\begin{align}\label{energy_estimate_nonlin_2}
			\operatorname{lim}_{t\rightarrow 0}t^{\frac{d}{2r}\left(\frac{p-1}{p}\right)}\lVert w(t)\rVert_{ L^{pr}_{rad}}&=0.    
		\end{align}
	\end{theorem}
	\begin{proof}
		We are looking for a solution of \eqref{nonlinear PDE} of the form
		\begin{align*}
			w(t)=\mathcal{T}[w_0,0](t)+\mathcal{T}[0,f(w)](t),
		\end{align*}
		$\mathcal{T}(g,f)$ being the solution map of \eqref{linear pde} with initial condition $g\in L^r_{rad}(\R^d)$ and singular forcing term $f$ as given in Lemma \ref{lem:linear PDE}. Let us start with some preliminary estimates needed later on. By H\"older inequality and definitions of the appearing terms, for $p>1+\frac{2}{d}$ we have the following inequalities
		\begin{align}\label{ineq_1_nonlin_PDE_sing}
			\lVert \lvert u'(t)\rvert^{p-1}\lvert w(t)\rvert\rVert_{L^{q_a}}\leq \lVert u'(t)\rVert^{p-1}_{L^r} \lVert w(t)\rVert_{L^r} \leq \frac{\eps^{p-1}}{t^{1-\frac{d(p-1)}{2r}}}\lVert w(t)\rVert_{L^r},
		\end{align}
		\begin{align}\label{ineq_2_nonlin_PDE_sing}
			\lVert \lvert u'(t)\rvert^{p-1}\lvert w(t)\rvert\rVert_{L^r}\leq \lVert u'(t)\rVert^{p-1}_{L^{pr}} \lVert w(t)\rVert_{L^{pr}} \leq \frac{\eps^{p-1}}{t}t^{\frac{d(p-1)}{2pr}}\lVert w(t)\rVert_{L^{pr}},
		\end{align}
		\begin{align}\label{ineq_11_nonlin_PDE_sing}
			\lVert \lvert u'(t)\rvert\lvert w(t)\rvert^{p-1}\rVert_{L^{q_a}}\leq \lVert u'(t)\rVert_{L^r} \lVert w(t)\rVert^{p-1}_{L^r} \leq \frac{\eps}{t^{\frac{1}{p-1}-\frac{d}{2r}}}\lVert w(t)\rVert_{L^r}^{p-1},
		\end{align}
		\begin{align}\label{ineq_12_nonlin_PDE_sing}
			\lVert \lvert u'(t)\rvert\lvert w(t)\rvert^{p-1}\rVert_{L^r}\leq \lVert u'(t)\rVert_{L^{pr}} \lVert w(t)\rVert^{p-1}_{L^{pr}} \leq \frac{\eps}{t^{\frac{1}{p-1}+\frac{d(p-2)}{2r}}}t^{\frac{d(p-1)^2}{2pr}}\lVert w(t)\rVert_{L^{pr}}^{p-1}.
		\end{align}
		If, moreover, $p>2$, then
		\begin{equation}\label{ineq_3_nonlin_PDE_sing}
			\lVert \lvert u'(t)\rvert^{p-2}\lvert w(t)\rvert^2\rVert_{L^{q_a}}\leq \lVert u'(t)\rVert^{p-2}_{L^r} \lVert w(t)\rVert^2_{L^r} \leq \frac{\eps^{p-2}}{t^{1-\frac{1}{p-1}-\frac{d(p-2)}{2r}}}\lVert w(t)\rVert^2_{L^r},
		\end{equation}
		
		\begin{equation}\label{ineq_4_nonlin_PDE_sing}
			\lVert \lvert u'(t)\rvert^{p-2}\lvert w(t)\rvert^2\rVert_{L^r}\leq \lVert u'(t)\rVert^{p-2}_{L^{pr}} \lVert w(t)\rVert^2_{L^{pr}} \leq \frac{\eps^{p-2}}{t^{1+\frac{d}{2r}-\frac{1}{p-1}}}t^{\frac{d(p-1)}{pr}}\lVert w(t)\rVert_{L^{pr}}^2,
		\end{equation}
		\begin{align}\label{ineq_5_nonlin_PDE_sing}
			\lVert \lvert \bar{u}(t)\rvert^{p-2}\lvert w(t)\rvert^2\rVert_{L^{q_a}}\leq \lVert \bar{u}(t)\rVert^{p-2}_{L^r} \lVert w(t)\rVert^2_{L^r} \leq \frac{\lVert \bar{U}\rVert_{L^r}^{p-2}}{t^{1-\frac{1}{p-1}-\frac{d(p-2)}{2r}}}\lVert w(t)\rVert^2_{L^r},
		\end{align}
		\begin{align}\label{ineq_6_nonlin_PDE_sing}
			\lVert \lvert \bar{u}(t)\rvert^{p-2}\lvert w(t)\rvert^2\rVert_{L^r}\leq \lVert \bar{u}(t)\rVert^{p-2}_{L^{pr}} \lVert w(t)\rVert^2_{L^{pr}} \leq \frac{\lVert \bar{U}\rVert_{L^{pr}}^{p-2}}{t^{1+\frac{d}{2r}-\frac{1}{p-1}}}t^{\frac{d(p-1)}{pr}}\lVert w(t)\rVert_{L^{pr}}^2,
		\end{align}
		\begin{align}\label{ineq_7_nonlin_PDE_sing}
			\lVert \lvert \bar{u}(t)\rvert^{p-2}\lvert u'(t)\rvert \lvert w(t)\rvert\rVert_{L^{q_a}}\leq \lVert \bar{u}(t)\rVert^{p-2}_{L^r}\lVert u'(t)\rVert_{L^r} \lVert w(t)\rVert_{L^r} \leq \frac{\eps\lVert \bar{U}\rVert_{L^{r}}^{p-2}}{t^{1-\frac{d(p-1)}{2r}}}\lVert w(t)\rVert_{L^r},
		\end{align}
		\begin{align}\label{ineq_8_nonlin_PDE_sing}
			\lVert \lvert \bar{u}(t)\rvert^{p-2}\lvert u'(t)\rvert \lvert w(t)\rvert\rVert_{L^r}\leq \lVert \bar{u}(t)\rVert^{p-2}_{L^{pr}}\lVert u'(t)\rVert_{L^{pr}} \lVert w(t)\rVert_{L^{pr}} \leq  \frac{\eps\lVert \bar{U}\rVert_{L^{pr}}^{p-2}}{t}t^{\frac{d(p-1)}{2pr}}\lVert w(t)\rVert_{L^{pr}}.
		\end{align}
		Lastly, since ${q_a}>\frac{(p-1)d}{2p}$ for $\theta$ small enough we have also ${q_a}>\frac{(p-1+\theta)d}{2(p+\theta)}$. If moreover such $\theta$ satisfies also $\theta<p(p-1)$ we can estimate \eqref{ineq_1_nonlin_PDE_sing}, \eqref{ineq_7_nonlin_PDE_sing} in the following way
		\begin{align}\label{ineq_going_to_0_1}
			\lVert \lvert u'(t)\rvert^{p-1}\lvert w(t)\rvert\rVert_{L^{q_a}}&\leq \lVert u'(t)\rVert^{p-1}_{L^{\frac{{q_a}(p-1)(p+\theta)}{p-1+\theta}}} \lVert w(t)\rVert_{L^{(p+\theta){q_a}}}\notag\\ & \leq \frac{\eps^{p-1}}{t^{1-\frac{d(p-1)}{2p{q_a}}}}\lVert w(t)\rVert_{L^r}^{1-\frac{\theta p}{(p+\theta)(p-1)}}\left(t^{\frac{d(p-1)}{2pr}}\lVert w(t)\rVert_{L^{pr}}\right)^{\frac{\theta p}{(p+\theta)(p-1)}}    
		\end{align}
		\begin{multline}\label{ineq_going_to_0_2}
			\lVert \lvert \bar{u}(t)\rvert^{p-2}\lvert u'(t)\rvert \lvert w(t)\rvert\rVert_{L^{q_a}}\leq \lVert \bar{u}(t)\rVert^{p-2}_{L^{\frac{{q_a}(p-1)(p+\theta)}{p-1+\theta}}}\lVert u'(t)\rVert_{L^{\frac{{q_a}(p-1)(p+\theta)}{p-1+\theta}}} \lVert w(t)\rVert_{L^{(p+\theta){q_a}}} \\  \leq \frac{\eps\lVert \bar{U}\rVert_{L^{\frac{{q_a}(p-1)(p+\theta)}{p-1+\theta}}}^{p-1} }{t^{1-\frac{d(p-1)}{2p{q_a}}}}\lVert w(t)\rVert_{L^r}^{1-\frac{\theta p}{(p+\theta)(p-1)}}\left(t^{\frac{d(p-1)}{2pr}}\lVert w(t)\rVert_{L^{pr}}\right)^{\frac{\theta p}{(p+\theta)(p-1)}}.    
		\end{multline}  
		Owing to our choice of $\theta$, we have that $\frac{{q_a}(p-1)(p+\theta)}{p-1+\theta}>q_c$ and $\lVert \bar{U}\rVert_{L_{rad}^{\frac{{q_a}(p-1)(p+\theta)}{p-1+\theta}}}<\infty$.
		After these preliminary computations, we can run a fixed point argument. Let $M=\lVert \mathcal{T}[w_0,0]\rVert_{Z^{e^{T}}}$, we denote by $B_{2M}\subseteq Z^{T'}$ the closed ball in $Z^{T'}$ for $T'>0$ with center $0$ and radius $2M$. We are looking for $\eps>0$ and $T'>0$ small enough such that 
		\begin{align*}
			\Gamma(w)=\mathcal{T}[w_0,0]+\mathcal{T}[0,f(w)]
		\end{align*}
		is a contraction on $B_{2M}$. First we need to show that $\Gamma$ maps $B_{2M}$ into itself. According to Lemma \ref{lemma_nonlinearity}, we have to split our analysis into cases $p\in (1+\frac{2}{d},2]$ and $p>2$. We start with the first case. Due to Lemma \ref{lemma_nonlinearity} with
		$x=\bar{u},\ y=u',\ z=u'-w$, we have
		\begin{align*}
			\lVert f(w)\rVert_{L^{q_a}_{rad}}&\lesssim_p \lVert  w\rVert_{L^r_{rad}}^p+\lVert \lvert u'\rvert^{p-1}\lvert w\rvert\rVert_{L^{q_a}_{rad}},\\  \lVert f(w)\rVert_{L^r_{rad}}&\lesssim_p \lVert w\rVert_{L^{rp}_{rad}}^p+\lVert \lvert u'\rvert^{p-1}\lvert w\rvert\rVert_{L^r_{rad}}.
		\end{align*}
		Therefore, by exploiting relations \eqref{ineq_1_nonlin_PDE_sing}, \eqref{ineq_2_nonlin_PDE_sing} and \eqref{ineq_going_to_0_1}, we obtain thanks to the fact that $w\in B_{2M}$
		\begin{align*}
			t^{1+\frac{d}{2r}-\frac{d}{2{q_a}}}\lVert f(w(t))\rVert_{L^{q_a}_{rad}}&\lesssim \eps^{p-1} M+\left(T'\right)^{1-\frac{d(p-1)}{2r}}M^{p},\\
			t\lVert f(w(t))\rVert_{L^r_{rad}}& \lesssim \eps^{p-1}M+\left(T'\right)^{1-\frac{d(p-1)}{2r}}M^p,\\
			\operatorname{lim}_{t\rightarrow 0} t^{1+\frac{d}{2r}-\frac{d}{2{q_a}}}\lVert f(w(t))\rVert_{L^{q_a}_{rad}}&=0,\quad \operatorname{lim}_{t\rightarrow 0} t\lVert f(w(t))\rVert_{L^r_{rad}}=0.
		\end{align*}
		Therefore, we can apply Lemma \ref{lem:linear PDE} to obtain
		\begin{align*}
			\lVert  \Gamma(w)\rVert_{Z^{T'}}&\leq M+C\left(\eps^{p-1} M+\left(T'\right)^{1-\frac{d(p-1)}{2r}}M^{p}\right)\\ & \leq 2M,
		\end{align*}
		whenever $\eps$ and $T'$ are small enough.
		Let us move to the case $p>2$, which is analogous, but more involved due to the structure of Lemma \ref{lemma_nonlinearity}. Applying this lemma for
		$x=\bar{u},\ y=u',\ z=u'-w$, we have
		\begin{align*}
			\lVert f(w)\rVert_{L^{q_a}_{rad}}\lesssim_p \, & \lVert w\rVert_{L^r_{rad}}^p+\lVert \lvert u'\rvert^{p-1}\lvert w\rvert\rVert_{L^{q_a}_{rad}} +\lVert\lvert u'\rvert^{p-2}\lvert w\rvert^2\rVert_{L^{q_a}_{rad}}+\lVert\lvert u'\rvert \lvert w\rvert^{p-1}\rVert_{L^{q_a}_{rad}},\\ &+\lVert\lvert \bar{u}\rvert^{p-2}\lvert u'\rvert\lvert w\rvert \rVert_{L^{q_a}_{rad}}+\lVert\lvert \bar{u}\rvert^{p-2}\lvert w\rvert^2 \rVert_{L^{q_a}_{rad}},\\  \lVert f(w)\rVert_{L^r_{rad}}\lesssim_p \, & \lVert w\rVert_{L^{pr}_{rad}}^p+\lVert \lvert u'\rvert^{p-1}\lvert w\rvert\rVert_{L^r_{rad}} +\lVert\lvert u'\rvert^{p-2}\lvert w\rvert^2\rVert_{L^r_{rad}}+\lVert\lvert u'\rvert \lvert w\rvert^{p-1}\rVert_{L^r_{rad}},\\ &+\lVert\lvert \bar{u}\rvert^{p-2}\lvert u'\rvert\lvert w\rvert \rVert_{L^r_{rad}}+\lVert\lvert \bar{u}\rvert^{p-2}\lvert w\rvert^2 \rVert_{L^r_{rad}}.
		\end{align*}
		Therefore, by exploiting relations \eqref{ineq_1_nonlin_PDE_sing}, \eqref{ineq_2_nonlin_PDE_sing}, \eqref{ineq_3_nonlin_PDE_sing}, \eqref{ineq_4_nonlin_PDE_sing}, \eqref{ineq_5_nonlin_PDE_sing}, \eqref{ineq_6_nonlin_PDE_sing}, \eqref{ineq_7_nonlin_PDE_sing}, \eqref{ineq_8_nonlin_PDE_sing}, \eqref{ineq_going_to_0_1}, \eqref{ineq_going_to_0_2}, and thanks to the fact that $w\in B_{2M}$, we obtain
		\begin{align*}
			t^{1+\frac{d}{2r}-\frac{d}{2{q_a}}}\lVert f(w(t))\rVert_{L^{q_a}_{rad}}&\lesssim\eps^{p-1} M+\left(T'\right)^{1-\frac{d(p-1)}{2r}}M^{p}+\eps^{p-2}\left(T'\right)^{\frac{1}{p-1}-\frac{d}{2r}}M^2\\ & \left. +\eps \left(T'\right)^{(p-2)\left(\frac{1}{p-1}-\frac{d}{2pq}\right)}M^{p-1}+\eps\lVert \bar{U}\rVert_{L^{r}}^{p-2}M\right.\\ &  +\lVert \bar{U}\rVert_{L^r_{rad}}^{p-2}\left(T'\right)^{\frac{1}{p-1}-\frac{d}{2r}}M^2 ,\\
			t\lVert f(w(t))\rVert_{L^r_{rad}}& \lesssim \eps^{p-1}M+\left(T'\right)^{1-\frac{d(p-1)}{2r}}M^p+ \eps^{p-2}\left(T'\right)^{\frac{1}{p-1}-\frac{d}{2r}}M^2\\ & \left. +\eps \left(T'\right)^{(p-2)\left(\frac{1}{p-1}-\frac{d}{2pq}\right)}M^{p-1}+\eps\lVert \bar{U}\rVert_{L^{r}}^{p-2}M\right.\\ &  +\lVert \bar{U}\rVert_{L^r_{rad}}^{p-2}\left(T'\right)^{\frac{1}{p-1}-\frac{d}{2r}}M^2,
		\end{align*}
		and
		\begin{align*}
			\operatorname{lim}_{t\rightarrow 0} t^{1+\frac{d}{2r}-\frac{d}{2{q_a}}}\lVert f(w(t))\rVert_{L^{q_a}_{rad}}&=0,\quad \operatorname{lim}_{t\rightarrow 0} t\lVert f(w(t))\rVert_{L^r_{rad}}=0.
		\end{align*}
		Consequently, we can apply Lemma \ref{lem:linear PDE} to obtain
		\begin{align*}
			\lVert  \Gamma(w)\rVert_{Z^{T'}}&\leq M+C\left(\eps^{p-1}M+\left(T'\right)^{1-\frac{d(p-1)}{2r}}M^p+ \eps^{p-2}\left(T'\right)^{\frac{1}{p-1}-\frac{d}{2r}}M^2\right.\\ & \left.\quad \quad\quad\quad\quad +\eps \left(T'\right)^{(p-2)\left(\frac{1}{p-1}-\frac{d}{2pq}\right)}M^{p-1}+\eps\lVert \bar{U}\rVert_{L^{r}}^{p-2}M\right.\\ & \left.\quad\quad\quad\quad\quad +\lVert \bar{U}\rVert_{L^r_{rad}}^{p-2}\left(T'\right)^{\frac{1}{p-1}-\frac{d}{2r}}M^2 \right)\\ & \leq 2M,
		\end{align*}
		for $\eps$ and $T'$ small enough.
		
		Now we show that $\Gamma$ is a contraction in $B_{2M}$, possibly reducing $\eps$ and $T'$. For $w_1,\ w_2\in B_{2M}$, we observe that
		\begin{align*}
			\Gamma(w_1)-\Gamma(w_2)=\mathcal{T}[0,f(w_1)-f(w_2)].
		\end{align*} We start again with the case $p\in (1+\frac{2}{d},2]$. Thanks to Lemma \ref{lemma_nonlinearity} with $x=\bar{u},\ y=u'-w_1,\ z=u'-w_2$, we get
		\begin{align*}
			\lVert f(w_1)-f(w_2)\rVert_{L^{q_a}_{rad}}\leq \, &  \lVert \lvert u' \rvert^{p-1}\lvert w_1-w_2\rvert \rVert_{L^{q_a}_{rad}}\\ &+ \left(\lVert w_1 \rVert^{p-1}_{L^r_{rad}}+\lVert w_2 \rVert^{p-1}_{L^r_{rad}}\right)\lVert w_1-w_2\rVert_{L^r_{rad}},\\
			\lVert f(w_1)-f(w_2)\rVert_{L^r_{rad}}\leq \, &  \lVert \lvert u' \rvert^{p-1}\lvert w_1-w_2\rvert \rVert_{L^r_{rad}}\\ &+ \left(\lVert w_1 \rVert^{p-1}_{L^{pr}_{rad}}+\lVert w_2 \rVert^{p-1}_{L^{pr}_{rad}}\right)\lVert w_1-w_2\rVert_{L^{pr}_{rad}}.
		\end{align*}
		Therefore, relations \eqref{ineq_1_nonlin_PDE_sing}, \eqref{ineq_2_nonlin_PDE_sing} and \eqref{ineq_going_to_0_1} imply, thanks to the fact that $w_1,w_2\in B_{2M}$
		\begin{align*}
			t^{1+\frac{d}{2r}-\frac{d}{2{q_a}}}\lVert f(w_1(t))-f(w_2(t))\rVert_{L^{q_a}_{rad}}&\lesssim\eps^{p-1} M+\left(T'\right)^{1-\frac{d(p-1)}{2r}}M^{p},\\
			t\lVert f(w_1(t))-f(w_2(t))\rVert_{L^r_{rad}}& \lesssim\eps^{p-1}M+\left(T'\right)^{1-\frac{d(p-1)}{2r}}M^p,\\
			\operatorname{lim}_{t\rightarrow 0} t^{1+\frac{d}{2r}-\frac{d}{2{q_a}}}\lVert f(w_1(t))-f(w_2(t))\rVert_{L^{q_a}_{rad}}&=0,\\ \operatorname{lim}_{t\rightarrow 0} t\lVert f(w_1(t))-f(w_2(t))\rVert_{L^r_{rad}}&=0.
		\end{align*}
		We can apply Lemma \ref{lem:linear PDE} to obtain
		\begin{align*}
			\lVert  \Gamma(w_1)-\Gamma(w_2)\rVert_{Z^{T'}}&\lesssim\left(\eps^{p-1}+\left(T'\right)^{1-\frac{d(p-1)}{2r}}M^{p-1}\right)\lVert w_1-w_2\rVert_{Z^{T'}}\\ & \leq \frac{1}{2}\lVert w_1-w_2\rVert_{Z^{T'}}
		\end{align*}
		for $\eps$ and $T'$ small enough. In case $p>2$, we get, thanks to Lemma \ref{lemma_nonlinearity} with $x=\bar{u},\ y=u'-w_1$, and $z=u'-w_2$, that
		\begin{align*}
			\lVert f(w_1)-f(w_2)\rVert_{L^{q_a}_{rad}}\leq\,&   \lVert \lvert u' \rvert^{p-1}\lvert w_1-w_2\rvert \rVert_{L^{q_a}_{rad}}\\ &+ \left(\lVert w_1 \rVert^{p-1}_{L^r_{rad}}+\lVert w_2 \rVert^{p-1}_{L^r_{rad}}\right)\lVert w_1-w_2\rVert_{L^r_{rad}}\\ & +\lVert \lvert \bar{u}\rvert^{p-2}\lvert u'\rvert \lvert w_1-w_2\rvert \rVert_{L^{q_a}_{rad}}\\ & +\lVert \lvert u'\rvert\left(\lvert w_1\rvert^{p-2}+\lvert w_2\rvert^{p-2}\right) \lvert w_1-w_2\rvert \rVert_{L^{q_a}_{rad}}\\ & +\lVert \lvert \bar{u}\rvert^{p-2}\left(\lvert w_1\rvert+\lvert w_2\rvert\right) \lvert w_1-w_2\rvert \rVert_{L^{q_a}_{rad}}\\&+\lVert \lvert u'\rvert^{p-2}\left(\lvert w_1\rvert+\lvert w_2\rvert\right) \lvert w_1-w_2\rvert \rVert_{L^{q_a}_{rad}},\\
			\lVert f(w_1)-f(w_2)\rVert_{L^r_{rad}}\leq\, &  \lVert \lvert u' \rvert^{p-1}\lvert w_1-w_2\rvert \rVert_{L^r_{rad}}\\ &+ \left(\lVert w_1 \rVert^{p-1}_{L^{pr}_{rad}}+\lVert w_2 \rVert^{p-1}_{L^{pr}_{rad}}\right)\lVert w_1-w_2\rVert_{L^{pr}_{rad}}\\ & +\lVert \lvert \bar{u}\rvert^{p-2}\lvert u'\rvert \lvert w_1-w_2\rvert \rVert_{L^r_{rad}}\\ & +\lVert \lvert u'\rvert\left(\lvert w_1\rvert^{p-2}+\lvert w_2\rvert^{p-2}\right) \lvert w_1-w_2\rvert \rVert_{L^r_{rad}}\\ & +\lVert \lvert \bar{u}\rvert^{p-2}\left(\lvert w_1\rvert+\lvert w_2\rvert\right) \lvert w_1-w_2\rvert \rVert_{L^r_{rad}}\\&+\lVert \lvert u'\rvert^{p-2}\left(\lvert w_1\rvert+\lvert w_2\rvert\right) \lvert w_1-w_2\rvert \rVert_{L^r_{rad}}.
		\end{align*}
		Therefore, by exploiting relations \eqref{ineq_1_nonlin_PDE_sing}, \eqref{ineq_2_nonlin_PDE_sing}, and \eqref{ineq_3_nonlin_PDE_sing}-\eqref{ineq_going_to_0_2}, we get, thanks to the fact that $w_1,w_2\in B_{2M}$, that
		\begin{align*}
			t^{1+\frac{d}{2r}-\frac{d}{2{q_a}}}\lVert f(w_1(t))-f(w_2(t))\rVert_{L^{q_a}_{rad}}\lesssim\, &\eps^{p-1} M+\left(T'\right)^{1-\frac{d(p-1)}{2r}}M^{p}\\ &+\eps^{p-2}\left(T'\right)^{\frac{1}{p-1}-\frac{d}{2r}}M^2\\ &     +\eps \left(T'\right)^{(p-2)\left(\frac{1}{p-1}-\frac{d}{2pq}\right)}M^{p-1}\\ &+\eps\lVert \bar{U}\rVert_{L^{r}}^{p-2}M  +\lVert \bar{U}\rVert_{L^r_{rad}}^{p-2}\left(T'\right)^{\frac{1}{p-1}-\frac{d}{2r}}M^2 ,\\
			t\lVert f(w_1(t))-f(w_2(t))\rVert_{L^r_{rad}} \lesssim \, &\eps^{p-1}M+\left(T'\right)^{1-\frac{d(p-1)}{2r}}M^p\\ &+ \eps^{p-2}\left(T'\right)^{\frac{1}{p-1}-\frac{d}{2r}}M^2\\ & +\eps \left(T'\right)^{(p-2)\left(\frac{1}{p-1}-\frac{d}{2pq}\right)}M^{p-1}\\ &+\eps\lVert \bar{U}\rVert_{L^{r}}^{p-2}M  +\lVert \bar{U}\rVert_{L^r_{rad}}^{p-2}\left(T'\right)^{\frac{1}{p-1}-\frac{d}{2r}}M^2 ,\end{align*}
		\begin{align*}
			&\operatorname{lim}_{t\rightarrow 0} t^{1+\frac{d}{2r}-\frac{d}{2{q_a}}}\lVert f(w_1(t))-f(w_2(t))\rVert_{L^{q_a}}=0,\\  &\operatorname{lim}_{t\rightarrow 0} t\lVert f(w_1(t))-f(w_2(t))\rVert_{L^r_{rad}}=0.
		\end{align*}
		We can then apply Lemma \ref{lem:linear PDE} to conclude that
		\begin{align*}
			\lVert  \Gamma(w_1)-\Gamma(w_2)\rVert_{Z^{T'}}&\lesssim \left(\eps^{p-1}+\left(T'\right)^{1-\frac{d(p-1)}{2r}}M^{p-1}+ \eps^{p-2}\left(T'\right)^{\frac{1}{p-1}-\frac{d}{2r}}M\right.\\ & \left.\quad\ +\eps \left(T'\right)^{(p-2)\left(\frac{1}{p-1}-\frac{d}{2pq}\right)}M^{p-2}+\eps\lVert \bar{U}\rVert_{L^{r}}^{p-2}\right.\\ & \left.\quad\ +\lVert \bar{U}\rVert_{L^r_{rad}}^{p-2}\left(T'\right)^{\frac{1}{p-1}-\frac{d}{2r}}M \right)\lVert w_1-w_2\rVert_{Z^{T'}}\\ & \leq \frac{1}{2}\lVert w_1-w_2\rVert_{Z^{T'}},
		\end{align*}
		for $\eps$ and $T'$ small enough. This completes the proof.
	\end{proof}
	\begin{remark}\label{quantification_T'}
		By choosing $\lVert w_0\rVert_{L^r_{rad}}$ small enough, we can, in fact, take $T'=e^{T}$.
	\end{remark}
	\section{Localization in $L^q$: Proof of Theorem \ref{Thm:main}}\label{sec:localization}
	\noindent Let $1\leq q<q_c.$ Let $\epsilon>0$ be such that
	\begin{align*}
		\frac{d(p-1)}{2q}-p<\epsilon<\frac{d(p-1)}{2q}-1.
	\end{align*}
	Then set $r=(p+\epsilon)q$. By Theorem \ref{Prop:Spectral_properties_L} there is $\bar{\alpha}>0$ such that for the corresponding expander $\bar{U}$ the operator  $L_{\bar{\alpha}}:\mathcal{D}(L_{\bar{\alpha}}) \subseteq L^{1,pr} \rightarrow L^{1,pr}$ admits a maximal positive eigenvalue $\la_{\bar{\alpha}}$ for which
	\begin{equation*}
		{\lambda}_{\bar{\alpha}}<\frac{1}{p-1}-\frac{d}{2r}.
	\end{equation*}
	Now, fix  $\bar{R}>0$ and define $u_0:\R^d \rightarrow \R$ by
	\begin{equation*}
		u_0(x):= \tilde{u}_0(x) \mathbf{1}_{[0,\bar{R}]}(|x|),
	\end{equation*} 
	where $\tilde{u}_0$ is given in \eqref{Eq:u_0}. Define $w_0:=\tilde{u}_0-u_0$. By setting $q_a=\frac{r}{p}$, $\hat{q}=1$, and $\hat{r}=pr$, we can now invoke Theorems \ref{thm:existence_ancient_solutions} and \ref{nonlinear_singular_PDE} to obtain 
	\begin{align*}
		u_1(t,x)&=\frac{1}{t^{\frac{1}{p-1}}}\bar{U}\left(\frac{|x|}{\sqrt{t}}\right) - w_1(t,x),\\ u_{2}(t,x)&=\frac{1}{t^{\frac{1}{p-1}}}\bar{U}\left(\frac{|x|}{\sqrt{t}}\right)+\frac{1}{t^{\frac{1}{p-1}}}\psi\left(\ln t,\frac{|x|}{\sqrt{t}}\right) - w_2(t,x),
	\end{align*}
	which are weak solutions of \eqref{PDE} on $(0,T')$ for some $T'>0$. In particular, one can easily check that $u_1, u_2\in L^p_{{loc}}((0,T')\times \R^d)$. We claim that they are, in fact, two different mild $L^q$-solutions on $[0,T')$ with the same initial datum $u_0$. By construction, $u_0$ is compactly supported, and therefore in $L^q(\R^d)$. It remains to show that $u_1, u_2\in C([0,T'),L^q(\R^d))$ and $u_1\neq u_2.$ We show the first property in the form of a lemma.
	
	\begin{lemma}\label{Continuity_uniform_bound}
		$u_1, u_2\in C([0,T'),L^q(\R^d))$ for each $1\leq q<q_c$.    
	\end{lemma}
	
	\begin{proof}
		We prove the lemma only for $u_2$, the other case being analogous and simpler. We therefore, for convenience, write $u$ (resp.~$w$) in place of $u_2$ (resp.~$w_2$). Let us introduce a radially symmetric cut off $\chi\in C^{\infty}_c(\R^d)$ by
		\begin{align*}
			\chi(x):=\begin{cases}
				1\quad \text{if } \lvert x\rvert\leq 1\\
				0\quad \text{if } \rvert x\rvert\geq 2,
			\end{cases}\quad \chi(x)\geq 0,\quad \chi(x)\leq \chi(x') \quad \text{if} \quad \lvert x\rvert\geq \lvert x'\rvert.
		\end{align*}
		Furthermore, for $x_0\in \R^d,\, R\geq 1$, define $\chi_{x_0, R}(x):=\chi\left(\frac{x-x_0}{R}\right).$ Due to the properties of $\bar{U},\psi, w$, we have that for each $x_0\in \R^d,\, R\geq 1$ \begin{align}\label{regularity_cutoff}
			u\chi_{x_0,R}\in C([0,T'],L^q(\R^d))\cap C((0,T'],L^r(\R^d)),
		\end{align}            
		and for each $t,t_0\in (0,T']$ with $t_0\leq t$, that
		\begin{align*}
			u(t)\chi_{x_0,R}= \, &P(t-t_0)u(t_0)\chi_{x_0,R}+\int_{t_0}^t P(t-s)\left(u(s)\Delta\chi_{x_0,R}-2\operatorname{div}(u(s)\nabla \chi_{x_0,R})\right) ds\\ & +\int_{t_0}^t P(t-s) \chi_{x_0,R} \lvert u(s)\rvert^{p-1}u(s) ds,
		\end{align*}
		in  $L^q(\R^d)$, where $P(t)$ is the heat semigroup on $\R^d$.
		Considering the $L^q$ norm of the equation above, by triangle inequality and the definition of $\chi_{x_0,R}$ we get easily  that
		\begin{align*}
			\norm{ u(t)}_{L^q(B_{R}(x_0))}\lesssim \, & \norm{u(t_0)\chi_{x_0,R}}_{L^q(\R^d)}+\int_{t_0}^t \left(1+\frac{1}{\sqrt{t-s}}\right) \norm{ u(s)}_{L^q(B_{2R}(x_0))}ds\\ & +\int_{t_0}^t\norm{ \chi_{x_0,R} \lvert u(s)\rvert^{p-1}u(s)}_{L^q(\R^d)} ds ,
		\end{align*}
		the hidden constant above being independent of $R$, $t_0,$ $t$ and $x_0$. Let us analyze further the last term.
		By interpolation and Young's inequality, since $r>pq$, we have
		\begin{align*}
			\norm{ \chi_{x_0,R} \lvert u(s)\rvert^{p-1}u(s)}_{L^q(\R^d)}= \, & \norm{ \chi^{1/p}_{x_0,R} u(s)}_{L^{pq}(\R^d)}^p\\ \leq \, & \norm{ \chi^{1/p}_{x_0,R} u(s)}_{L^{q}(\R^d)}^{\frac{r-pq}{r-q}}\norm{  u(s)}_{L^{r}(\R^d)}^{\frac{r(p-1)}{r-q}}\\  \lesssim \, & \norm{ \chi^{1/p}_{x_0,R} u(s)}_{L^{q}(\R^d)}+\norm{  u(s)}_{L^{r}(\R^d)}^{\frac{r}{q}}\\ \lesssim \, & \norm{  u(s)}_{L^{q}(B_{2R}(x_0))}+\frac{1}{s^{\frac{r}{q}\left(\frac{1}{p-1}-\frac{d}{2r}\right)}}\left(\norm{\bar{U}}_{L^r(\R^d)}^{\frac{r}{q}}+\eps^{\frac{r}{q}}\right)\\ & +\norm{w}_{C([0,T'],L^r(\R^d))}^{\frac{r}{q}}.
		\end{align*}
		Therefore 
		\begin{align*}
			\norm{ u(t)}_{L^q(B_{R}(x_0))}\lesssim \,& \norm{u(t_0)\chi_{x_0,R}}_{L^q(\R^d)}+\int_{t_0}^t \left(1+\frac{1}{\sqrt{t-s}}\right) \norm{ u(s)}_{L^q(B_{2R}(x_0))}ds\\ & +\int_{t_0}^t 1+\frac{1}{s^{\frac{r}{q}\left(\frac{1}{p-1}-\frac{d}{2r}\right)}} ds.
		\end{align*}
		Since $u_0\in L^q(\R^d)$, $\frac{r}{q}\left(\frac{1}{p-1}-\frac{d}{2r}\right)<1$, and the relation \eqref{regularity_cutoff} holds, we can let $t_0\rightarrow 0$ to get
		\begin{align*}
			\norm{ u(t)}_{L^q(B_{R}(x_0))}&\lesssim \left(\norm{u_0}_{L^q(\R^d)}+1\right)+\int_{0}^t \left(1+\frac{1}{\sqrt{t-s}}\right) \norm{ u(s)}_{L^q(B_{2R}(x_0))}ds.   
		\end{align*}
		Now, let us introduce the function \begin{align*}
			u_R(t)=\sup_{x_0\in \R^d}\norm{ u(t)}_{L^q(B_{R}(x_0))}.
		\end{align*}
		Obviously
		\begin{align*}
			\norm{ u(t)}_{L^q(B_{2R}(x_0))}\lesssim u_R(t)
		\end{align*}
		and 
		\begin{align*}
			\norm{ u(t)}_{L^q(B_{R}(x_0))}  &\lesssim \left(\norm{u_0}_{L^q(\R^d)}+1\right)+\int_{0}^t \left(1+\frac{1}{\sqrt{t-s}}\right) u_R(s)ds.  
		\end{align*}
		Taking the supremum in $x_0$ of the expression above and applying Gr\"onwall's inequality, we get 
		\begin{align*}
			u_R(t)\lesssim \norm{u_0}_{L^q(\R^d)}+1,
		\end{align*}
		for all $t \in [0,T']$.
		By letting $R\rightarrow +\infty$, we get that $u\in L^{\infty}((0,T),L^q(\R^d))$. In order to show the continuity of $u(t)$ in $L^q$, we use again the mild formulation. Namely, for each $0\leq t_1\leq t_2\leq T'$, we have that
		\begin{align*}
			u(t_2)\chi_{x_0,R}-u(t_1)&=\left(P(t_2-t_1)u(t_1)\chi_{x_0,R}-u(t_1)\right)\\ &+\int_{t_1}^{t_2} P(t-s)\left(u(s)\Delta\chi_{x_0,R}-2\operatorname{div}(u(s)\nabla \chi_{x_0,R})\right) ds\\ & +\int_{t_1}^{t_2} P(t-s) \chi_{x_0,R} \lvert u(s)\rvert^{p-1}u(s) d,
		\end{align*}
		in $L^q(\R^d)$.   
		Taking the $L^q$-norm of the expression above and letting $R\rightarrow +\infty$, we obtain, by analogous considerations to the ones employed to obtain the uniform bound on the $L^q$-norm
		\begin{align*}
			\norm{u(t_2)-u(t_1)}_{L^q(\R^d)} = \, &\limsup_{R\rightarrow +\infty}\norm{u(t_2)\chi_{x_0,R}-u(t_1)}_{L^q(\R^d)}\\  \leq \, & \norm{P(t_2-t_1)u(t_1)-u(t_1)}_{L^q(\R^d)}\\ & +\sup_{R\geq 1}\norm{\int_{t_1}^{t_2} P(t-s)\left(u(s)\Delta\chi_{x_0,R}-2\operatorname{div}(u(s)\nabla \chi_{x_0,R})\right)}_{L^q(\R^d)}\\ & +\sup_{R\geq 1}\norm{\int_{t_1}^{t_2} P(t-s) \chi_{x_0,R} \lvert u(s)\rvert^{p-1}u(s) ds }_{L^q(\R^d)}\\  \lesssim\, & \norm{P(t_2-t_1)u(t_1)-u(t_1)}_{L^q(\R^d)}\\ & +\int_{t_1}^{t_2} \left(1+\frac{1}{\sqrt{t_2-s}}\right) \norm{ u(s)}_{L^q(\R^d)}ds+\int_{t_1}^{t_2} 1+\frac{1}{s^{\frac{r}{q}\left(\frac{1}{p-1}-\frac{d}{2r}\right)}} ds.
		\end{align*}
		Since we already proved that $u\in L^{\infty}((0,T),L^q(\R^d)),$ the claimed continuity follows from the last inequality.
	\end{proof}
	It remains to show that $u_1\neq u_2$. Let us consider the $L^r$-norm of $u_1-u_2$. Thanks to Theorems \ref{thm:existence_ancient_solutions} and \ref{nonlinear_singular_PDE}, and relation \eqref{Item instabilty}, we get that
	\begin{align*}
		\norm{u_1(t)-u_2(t)}_{L^r(\R^d)}&\geq {t^{-\left(\frac{1}{p-1}-\frac{d}{2r}\right)}}\norm{\psi(\ln t)}_{L^r(\R^d)}-\norm{w_1-w_2}_{C([0,T'],L^r(\R^d))}\\ &\gtrsim t^{-\left(\frac{1}{p-1}-\frac{d}{2r}-\lambda_{\bar{\alpha}}\right)}\frac{\norm{\bar{U}^{lin}}_{L^{r}(\R^d)}}{2}-1\rightarrow +\infty
	\end{align*}
	as $t\rightarrow 0$.

	\bibliography{refs.bib}
	\bibliographystyle{plain}
\end{document}